\documentclass[a4paper,12pt,reqno]{article}
 \usepackage[english]{babel}
  \usepackage{amsmath,amsfonts,amssymb,amsthm,bbm}
   \usepackage{graphics,epsfig,psfrag} 
                            \usepackage{stmaryrd}
   \usepackage{paralist,subfigure}
   \usepackage[dvipsnames]{xcolor}
   \usepackage{hyperref} 

    \usepackage[active]{srcltx} 
    \usepackage{color,soul} 
    \usepackage[latin1]{inputenc}
    \usepackage[OT1]{fontenc}
    \usepackage{authblk}
\allowdisplaybreaks
\newtheorem{theorem}{Theorem}[section]

\newtheorem{lemma}[theorem]{Lemma}
\newtheorem{prop}[theorem]{Proposition}

\theoremstyle{definition}

\newtheorem{req}{Remark}

\DeclareMathOperator\inter{int}

\def\N{\mathbb{N}}

\def\R{\mathbb{R}}

\let\O=\Omega
\let\e=\varepsilon

\let\t=\tilde
\let\ol=\overline

\let\.=\cdot
\let\0=\emptyset

\let\mc=\mathcal

\def\1{\mathbbm{1}}

\def\pe{principal eigenvalue}
\def\pf{principal eigenfunction}

\def\l{\lambda_1}

\def\thm#1{Theorem~\ref{thm:#1}}

\newenvironment{formula}[1]{\begin{equation}\label{#1}}
                       {\end{equation}\noindent}

\def\Fi#1{\begin{formula}{#1}}
\def\Ff{\end{formula}\noindent}

\renewcommand{\theequation}{\arabic{section}.\arabic{equation}}

\newcommand{\rest}[2]{\!\left.\!\hspace{1pt}#1\right|_{#2}}

\title{\bf Generalized principal eigenvalues for heterogeneous road-field systems}

\author[1]{Henri {\sc Berestycki}}
\author[1, 2]{Romain {\sc Ducasse}}
\author[1]{Luca {\sc Rossi}}
\affil[1]{Ecole des Hautes \'Etudes en Sciences Sociales, PSL University,  
		Centre d'Analyse et Math\'ematiques Sociales, 54 boulevard Raspail 
		75006 Paris, France}
\affil[2]{Aix Marseille Université,  
		I2M, Marseille}	

\begin{document}

\date{}
\maketitle



\begin{abstract}
This paper develops the notion and properties of the generalized
principal  eigenvalue for an elliptic  system coupling an equation in a plane with one on a line in this plane, together with boundary conditions that express exchanges taking place between the plane and the line. This study is motivated by the
 reaction-diffusion model introduced by H. Berestycki, J.-M. Roquejoffre and L. Rossi \cite{BRR1} to describe the effect  on biological invasions of networks with fast diffusion imbedded in a field. Here we study the eigenvalue associated with heterogeneous generalizations of this model.  In a forthcoming work \cite{BDR2} we show that persistence or extinction of the associated nonlinear evolution equation is fully accounted for by this generalized eigenvalue. A key element in  the proofs is a new Harnack inequality that we establish for these systems and which is of independent interest.
\end{abstract}

\section{Introduction}

\subsection{A model problem}
To illustrate the type of problems we investigate in this paper, let us describe the particular case of a half-plane $\Omega:= \{ (x,y) \in \R^2, \ y>0\}$ whose boundary we identify with $\R$. Let $a(x,y), \mu(x), \nu (x)$ be given, say bounded, functions defined in $\Omega$ 
and~$\R$ respectively and $d, D >0 $ be given. We consider the eigenvalue problem of the type: 

\begin{equation}\label{type}
\left\{
\begin{array}{rll}
-D \varphi^{\prime\prime} -  \nu(x) \psi\vert_{y=0} + \mu(x) \varphi &= \lambda \varphi,  \quad  \, &x\in \mathbb{R}, \\
-d\Delta \psi + a(x,y) \psi &=  \lambda \psi ,  \quad   \, &(x,y)\in \Omega := \mathbb{R}\times \R^+, \\
- d\partial_{y}\psi\vert_{y=0} - \mu (x) \varphi + \nu(x) \psi\vert_{y=0} &= 0, \quad & x\in \mathbb{R}= \partial \Omega.
\end{array}
\right.
\end{equation}
This systems involves two difficulties. First, it is set in an unbounded domain. Then it couples 
two equations set in domains of different dimensions. 

In this paper, our first task is to give a definition of the {\em generalized principal eigenvalue} for such a problem. This is done in the spirit of \cite{BNV, BR2}. We then establish a Harnack inequality for this type of systems. Such an inequality is at the core of the various following proofs. It first allows us to show the existence of a {\em generalized principal eigenfunction} in the unbounded domain, that is,
a pair of positive solutions~$(\varphi, \psi)$  of~\eqref{type}.
Using this fact, we can identify this principal eigenvalue as the limit of eigenvalues in bounded domain approximations. Lastly we derive a Rayleigh-Ritz type formula that yields this eigenvalue. Thanks to these results we also derive various qualitative properties of the generalized eigenvalue and eigenfunction. 

The system \eqref{type} is a particular case of the more general class of problems that we consider here. We state precise and general results in Section~\ref{Our model} below.

This class of eigenvalue problems is at the core of the understanding of models in ecology that we refer to  as ``road-field'' models. In the next section we describe this motivation.

\subsection{The road-field model}
Reaction-diffusion equations and systems arise in a wide range of applications, in ecology (dynamics of populations, spreading of invasive species,...), physics (combustion theory), and, more recently, in social sciences (riot propagations).  The iconic example is the Fisher-Kolmogorov-Petrovski-Piskunov equation (see the original papers \cite{F} for the biological motivations and \cite{KPP} for a mathematical study):
\begin{equation}\label{eq kpp}
\partial_{t}u -d\Delta u = f(u), \quad  t>0 ,\ x \in \mathbb{R}^{N}.
\end{equation}
The function $f$ is supposed to satisfy $f(0)=f(1)=0$, together with
$f>0$ on $(0,1)$,  and $s \mapsto \frac{f(s)}{s}$ is non-increasing (hence, $f$ is differentiable at $0^+$ with $f^{\prime}(0^+)>0$). More generally, we assume the KPP condition, that is: $f^{\prime}(0)s>f(s)>0$, for $s\in (0,1]$. A typical example is the logistic nonlinearity $f(s)=s(1-s)$.
In this setting, the solution $u(t,x)$ of \eqref{eq kpp} can be viewed as being the density of a population distributed on $\R^N$ subject to diffusion as arising from random motion of individuals, and to reproduction and mortality, as accounted by the reaction term $f$.

Basic results \cite{AW, KPP} state that any solution $u(t,x)$ of \eqref{eq kpp} arising from a non-negative, non-trivial, compactly supported initial datum, satisfies the \emph{invasion} property:
$$
u(t,x) \underset{t \to +\infty}{\longrightarrow} 1 \quad \text{ locally uniformly in }\ x.
$$
Moreover, one can define the \emph{speed of invasion} as a quantity $c_{KPP} \geq 0$, such that
$$
\forall \, c > c_{KPP}, \quad \displaystyle{\sup_{\vert x \vert \geq ct}} u(t,x) \underset{ t \to +\infty}{\longrightarrow} 0 
$$
$$
\forall \,  c \in [0, c_{KPP}), \quad \displaystyle{\inf_{\vert x \vert \leq ct  }} u(t,x)  \underset{ t \to +\infty}{\longrightarrow} 1.
$$
A classical result of \cite{AW} is that the speed of invasion can be explicitly computed in this case: $c_{KPP}=2\sqrt{d f^{\prime}(0^+)}$.

\medskip
Building on the Fisher-KPP model, H. Berestycki, J.-M. Roquejoffre and L.~Rossi~\cite{BRR1} introduced a system devised to describe the influence of a line with fast diffusion on biological invasions, that we call the \emph{road} in the sequel. The idea is to divide the population into two groups, considering two densities: $u$ denotes the density of population on the road, while $v$ stands for the density of population outside of the road (referred to as the \emph{field}). The model of \cite{BRR1} in a half-plane reads as follows:
\begin{equation}\label{fr normal}
\left\{
\begin{array}{rll}
\partial_{t}u-D \partial_{x}^{2}u &= \nu v\vert_{y=0} - \mu u,  \quad &t >0 \, , \, x\in \mathbb{R}, \\
\partial_{t}v-d\Delta v &= f(v),  \quad & t >0 \, , \, (x,y)\in \mathbb{R}\times \R^+, \\
- d\partial_{y}v\vert_{y=0} &= \mu u-\nu v\vert_{y=0}, \quad &t >0 \, , \, x\in \mathbb{R}.
\end{array}
\right.
\end{equation}
The second equation reflects the fact that, in the field $\R\times \R^+$, the evolution of the population is given by the usual Fisher-KPP equation \eqref{eq kpp} with diffusivity $d>0$. On the road $\R\times \{ 0 \}$, the population is subject to a diffusion equation with diffusivity $D>0$, which is a priori different from $d$. The third equation translates the exchanges between the road and the field: the parameter $\mu>0$ represents the rate of individuals on the road leaving it to reach the field, while the parameter $\nu>0$ represents the rate of individuals in the field to enter the road.

It is proven in \cite{BRR1} that the road enhances the speed of invasion in the field if the diffusion on the road $D$ is large enough, compared to $d$. More precisely, first, for any compactly supported non-negative, non-trivial initial datum $(u_{0},v_{0})$, spreading in the $x$-direction occurs with some speed $c_{BRR}\geq c_{KPP}$, depending on the parameters $d,D,\mu,\nu$. This speed satisfies $ c_{BRR}>c_{KPP}$ if and only if $D>2d$.

\medskip
In this paper, we introduce some tools necessary for the analysis of generalizations of the field-road system \eqref{fr normal}. We present such generalizations in the two next sections before getting to the more general case studied in this paper.

\subsection{Effect of a road on an ecological niche}

The system \eqref{fr normal} is homogeneous: every place in the field is equally suitable for the survival of the population. In many applications, this homogeneity hypothesis is not verified, and it is of the essence to take into account heterogeneities of the medium. The simplest way to do this is to consider coefficients that depend on the space variables. In the paper \cite{BDR2}, we study a model designed to account for the effect of a road on a species that can reproduce efficiently only in a bounded portion of the field (we say that the species has a \emph{bounded ecological niche}). In its simplest form, the model of \cite{BDR2} reads
\begin{equation}\label{fr cc}
\left\{
\begin{array}{rll}
\partial_{t}u-D\partial_{x}^{2}u &= \nu v\vert_{y=0} - \mu u,  \quad &t >0 \, , \, x\in \mathbb{R}, \\
\partial_{t}v-d\Delta v &= f(x,y,v),  \quad & t >0 \, , \, (x,y)\in \R \times \R^+, \\
- d\partial_{y}v\vert_{y=0} &= \mu u-\nu v\vert_{y=0}, \quad &t >0 \, , \, x\in \mathbb{R},
\end{array}
\right.
\end{equation}
where $f$ is a KPP-type nonlinearity such that $\limsup_{\vert (x,y) \vert \to +\infty}f_{v}(x,y,0) <0$.

The study of the heterogeneous model \eqref{fr cc} differs greatly from that of the homogeneous model \eqref{fr normal}. It is based on the notion of \emph{generalized principal eigenvalue} for the linearization near $(u,v) = (0,0)$ of the stationary system associated to \eqref{fr cc}, that is (where $f_{v}(x,y,0)$ denotes the derivative of $v\mapsto f(x,y,v)$ evaluated at $v=0$):
\begin{equation}\label{linear syst fr}
\left\{
\begin{array}{rll}
-Du^{\prime\prime} - \nu v\vert_{y=0} + \mu u \, &= \lambda u,  \quad &x\in \mathbb{R}, \\
-d\Delta v - f_v(x,y,0)v \, &= \lambda v,  \quad &(x,y)\in \R \times \R^+, \\
- d\partial_{y}v\vert_{y=0} &= \mu u-\nu v\vert_{y=0}, \quad &x\in \mathbb{R}.
\end{array}
\right.
\end{equation}
Observe that this system is the one \eqref{type} above with $ a(x, y) = -  f_v(x,y,0)v$. It   is set on an unbounded domain. When working with elliptic operators on smooth bounded domains, the Krein-Rutman theorem (see, e.g., \cite{KR}) gives the existence of a \emph{principal eigenvalue}. It is the smallest eigenvalue, and it is characterized by the existence of a positive eigenfunction. In our context, this notion is more delicate since the domain is unbounded and the system couples equations set in different dimensions. The definition of a notion of generalized principal eigenvalue for systems such as \eqref{linear syst fr} and the study of its properties are precisely the subject of the present work.

The generalized principal eigenvalue $\l $ for \eqref{linear syst fr} is defined as follows:
\begin{multline*}
\lambda_{1}:=\sup\Big\{\lambda \in \mathbb{R} \ : \ \exists (u,v) \geq (0,0),\ (u,v)\not\equiv (0,0) \ \text{ such that } \ \\ 
\quad\quad\quad Du^{\prime\prime}+\nu v\vert_{y=0} - \mu u + \lambda u \leq 0,\  d\Delta v + f_{v}(x,y,0)v+\lambda v  \leq 0 ,\\ \text{ and }\ d\partial_{y}v\vert_{y=0} + \mu u-\nu v\vert_{y=0} \leq 0 \Big\}.
\end{multline*}

The usefulness of $\l$ is that it completely characterizes the long-time behavior of~\eqref{fr cc}. More precisely, we prove in our forthcoming paper \cite{BDR2} that \emph{invasion} occurs (that is, the solutions converge to a positive stationary solution) if $\l<0$ while \emph{extinction} occurs (i.e., every solution goes to zero) if $\l \geq 0$. 

Then, using some of the properties of $\lambda_1$ derived in the present paper, we deduce that the presence of a road can never have a positive effect on the survival of the population, but on the contrary it might have a deleterious effect leading to the extinction of the population.

Two crucial properties of the generalized principal eigenvalue are:
\begin{enumerate}[(1)]
\item The existence of a positive eigenfunction $(u,v)$ associated with $\l$.

\item The generalized principal eigenvalue $\l$ is approximated by the principal eigenvalue for the same problem on \emph{truncated domains}, with Dirichlet boundary conditions on the truncation.
\end{enumerate}

Both properties are well-known for the standard principal eigenvalue of an elliptic operator, but their validity in the context of the road-field systems such as (1.4) was not known up to now. They are two of the main results of the present paper, see 
Theorems \ref{thm:gpef} and \ref{thm:truncated} below. The key ingredient to derive them is a \emph{Harnack inequality} for road-field systems: for any compact sets $K_1 \subset \R$ and $K_2 \subset \R \times [0,+\infty)$, there is a constant $C>0$ such that any positive solution $(u,v)$ of \eqref{linear syst fr} satisfies
$$
\max\left\{ \sup_{K_1}u   ,   \sup_{K_2}v \right\} \leq 
C\min\left\{\inf_{K_1}u  ,  \inf_{K_2}v \right\}.
$$

Let us mention that a notion of generalized principal eigenvalue for road-field models has been recently
 introduced by T. Giletti, L. Monsaingeon and M. Zhou in~\cite{GMZ}. They prove that the properties $(1)$ and $(2)$ mentioned above hold for the specific case of road-field models with periodic exchanges coefficients.

\subsection{Extension to the case of a road through a field}

In the models \eqref{fr normal} and \eqref{fr cc}, the field is a half-plane and the road is its boundary.  It is natural to consider a field that is the whole plane, with the road being, say, the line $\R\times\{0\}$, dividing the plane into an upper field and a lower field. In symmetric situations such as \eqref{fr cc}, this is tantamount to considering the system in the half-plane. However, for general non-symmetric dependence with respect to the vertical variable, the description of the model requires $3$ equations: one for each side of the field, one for the road, with two boundary conditions that are exchanges conditions, one for each of the two connections road-field. The densities of population in the upper and lower field are denoted by $v_1$, $v_2$ respectively. We do not assume that the exchanges between the road and the field are the same for each side of the field. The system reads
\begin{equation}\label{pas sym}
\left\{
\begin{array}{rll}
\partial_{t}u-D\partial_{x}^{2}u&= \nu_1 v_1\vert_{y=0}+\nu_2 v_2\vert_{y=0}-(\mu_1 + \mu_2) u,  \quad &t >0 \, , \, x\in \mathbb{R}, \\
\partial_{t}v_i-d_i\Delta v_i&= f_i(x,y,v_i),  \quad & t >0 \, , \, (x,y)\in \mathbb{R}\times \mathbb{R}^{+}, \\
-d_i\partial_{y}v_i\vert_{y=0} &= \mu_i u_i-\nu_i v_i\vert_{y=0}, \quad &t >0 \, , \, x\in \mathbb{R}.
\end{array}
\right.
\end{equation}
In the case where the exchanges and the nonlinearity are symmetric, i.e., $\mu_1=\mu_2$, $\nu_1=\nu_2$ and $f_1 = f_2$, then \eqref{pas sym} reduces \eqref{fr cc}.

The long-time behavior of system \eqref{pas sym} is also characterized by the sign of the generalized principal eigenvalue of its linearization near $(u,v_1,v_2)=(0,0,0)$. This eigenvalue is defined by the following formula, that generalizes the previous one:
\begin{multline*}
\lambda_{1}:=\sup\Big\{\lambda \in \mathbb{R} \ : \ \exists (u,v_1,v_2) \geq (0,0,0),\ (u,v_1,v_2)\not\equiv (0,0,0) \ \text{ such that } \ \\ 
 Du^{\prime\prime}+\nu_1 v_1\vert_{y=0}+\nu_2 v_2\vert_{y=0}-(\mu_1 + \mu_2) u+ \lambda u \leq 0,\\ 
  d_i\Delta v_i + f_{i,v}(x,0)v_i+\lambda v  \leq 0 ,\  d_i\partial_{y}v_i\vert_{y=0} +\mu_i u_i-\nu_i v_i\vert_{y=0} \leq 0 \ \text{ for }\ i=1,2 \Big\}.
\end{multline*}

\subsection{Model with climate change}

In the paper \cite{BDR2}, where we study the effect of a road on an ecological niche, we also take consider the possibility for the niche to move as time goes by. This reflects the effect of a climate change (see \cite{BDNZ}). This leads us to consider the system
\begin{equation*}
\left\{
\begin{array}{rll}
\partial_{t}u-D\partial_{x}^{2}u  &= \nu_1 v_1\vert_{y=0}+\nu_2 v_2\vert_{y=0}-(\mu_1 + \mu_2) u,  \quad &t >0 \, , \, x\in \mathbb{R}, \\
\partial_{t}v_i-d_i\Delta v_i  &= f_i(x-ct,y,v_i),  \quad & t >0 \, , \, (x,y)\in \mathbb{R}\times \mathbb{R}^{+}, \\
-d_i\partial_{y}v_i\vert_{y=0} &= \mu_i u_i-\nu_i v_i\vert_{y=0}, \quad &t >0 \, , \, x\in \mathbb{R},
\end{array}
\right.
\end{equation*}
which is again a generalization of \eqref{fr normal}, \eqref{fr cc} and \eqref{pas sym}, which is recovered in the particular case $c=0$ of this system. There, the heterogeneities move with constant speed $c$ in the direction of the road. That is, $c$ represents the speed of a climate change. Rewriting this system in the frame that keeps pace with the shifting environment, we obtain
\begin{equation}\label{fr gen 1}
\left\{
\begin{array}{rll}
\partial_{t}u-D\partial_{x}^{2}u -c \partial_{x}u &= \nu_1 v_1\vert_{y=0}+\nu_2 v_2\vert_{y=0}-(\mu_1 + \mu_2) u,  \quad &t >0 \, , \, x\in \mathbb{R}, \\
\partial_{t}v_i-d_i\Delta v_i - c \partial_{x}v_i &= f_i(x,y,v_i),  \quad & t >0 \, , \, (x,y)\in \mathbb{R}\times \mathbb{R}^{+}, \\
-d_i\partial_{y}v_i\vert_{y=0} &= \mu_i u_i-\nu_i v_i\vert_{y=0}, \quad &t >0 \, , \, x\in \mathbb{R}.
\end{array}
\right.
\end{equation}
We denote by $a_i(x,y)$ the derivative of $v\mapsto f_{i}(x,y,v)$ evaluated at $v=0$. The linearization near $(u,v_1,v_2) = (0,0,0)$ of \eqref{fr gen 1} reads
\begin{equation}\label{lin 1}
\left\{
\begin{array}{rll}
-Du^{\prime\prime} -c u^{\prime} &= \nu_1 v_1\vert_{y=0}  + \nu_2 v_2\vert_{y=0}   -(\mu_1 + \mu_2) u,  \quad &x\in \mathbb{R}, \\
-d_i\Delta v - c\partial_{x}v_i &= a_i(x,y)v_i,  \quad &(x,y)\in \mathbb{R}\times \mathbb{R}^{+}, \\
-d_i\partial_{y}v_i\vert_{y=0} &= \mu_i u-\nu_i v_i\vert_{y=0}, \quad &x\in \mathbb{R}.
\end{array}
\right.
\end{equation}
As before, the generalized principal eigenvalue of this system is defined by
\begin{multline*}
\lambda_{1}:=\sup\Big\{\lambda \in \mathbb{R} \ : \ \exists (u,v_1,v_2) \geq (0,0,0),\ (u,v_1,v_2)\not\equiv (0,0,0) \ \text{ such that } \ \\ 
 Du^{\prime\prime}+cu^{\prime}+\nu_1 v_1\vert_{y=0}+\nu_2 v_2\vert_{y=0}-(\mu_1 + \mu_2) u+ \lambda u \leq 0,\\ 
  d_i\Delta v_i + c\partial_{x}v_i+a_i(x,y)v_i+\lambda v_i  \leq 0 ,\  d_i\partial_{y}v_i\vert_{y=0} +\mu_i u_i-\nu_i v_i\vert_{y=0} \leq 0 \ \text{ for }\ i=1,2 \Big\}.
\end{multline*}
Even in this general setting, we derive the existence of a generalized principal eigenfunction 
(Theorem \ref{thm:gpef}) and the approximation of the generalized principal eigenvalue by the principal eigenvalues of the same system restricted to truncated domains (Theorem~\ref{thm:truncated}). 
Moreover, we prove in \cite{BDR2} that the sign of the generalized principal eigenvalue characterizes the long-time behavior of \eqref{fr gen 1}. Therefore, understanding the impact of a climate change on a population, combined with the presence of a road, is tantamount to study how $\l$ depends on the coefficients of the system. We prove in \cite{BDR2} that there are two critical speeds $0\leq c_{\star}\leq c^{\star}$ such that persistence occurs if $\vert c \vert < c_{\star}$, i.e., if the climate change is not too fast, and extinction occurs if $\vert c \vert \geq c^{\star}$, that is, if the climate change is too fast. We also prove that adding a line with fast diffusion can decrease $\l$. As a consequence, we exhibit situations where adding a road allows a population to persist in face of a faster climate change than it would without the road. The keystone to derive these results is a Rayleigh-Ritz characterization of $\l$, see Proposition \ref{variational} below.

\subsection{Related works and perspectives}

The model \eqref{fr normal} has been extended in several ways. The article \cite{BRR2} studies the influence of drift terms and mortality on the road. In a further paper \cite{BRR3}, H. Berestycki, J.-M. Roquejoffre and L. Rossi \cite{BRR3} computes the spreading speed in all directions of the field. The paper \cite{GMZ} treats the case where the exchanges coefficients $\mu, \nu $ are not constant but periodic in $x$, \cite{P1, P2} studieds non-local exchanges and \cite{Diet1} considers a combustion nonlinearity instead of the KPP one. The articles \cite{BCRR_sem,BCRR} study the effect of non-local diffusion. Different geometric situations are considered in \cite{RTV, Dcurved}. The first one treats the case where the field is a cylinder with its boundary playing the role of the road, and the second one studies the case where the road is curved.

The system \eqref{fr cc}, describing the interaction between a line with fast diffusion and climate change is inspired by the climate change model of \cite{BDNZ} (see also \cite{BRforcedspeed}).

The notion of generalized principal eigenvalue was introduced by H. Berestycki, L.~Nirenberg and S. Varadhan in \cite{BNV} to study elliptic operators on non-smooth domains, with Dirichlet boundary conditions. It was later used by two of the authors of the present article \cite{BR1, BR2} to study elliptic operators in unbounded domains, also with Dirichlet boundary conditions. Using this notion in \cite{BHReigen}, H. Berestycki, F. Hamel and L. Rossi show that the generalized principal eigenvalue plays a crucial role in the study of reaction-diffusion equations, to characterize the validity of the maximum principle, the existence of stationary solutions, as well as the existence of speeds of invasions and of traveling fronts.

Let us mention some perspective and natural extensions for this work. First, the results presented here could be used to study KPP reaction-diffusion equations on domains with \emph{junctions}. In their simplest setting, these are domains consisting of metric graphs, that is, vertices in $\R^N$ connected by straight segment of lines. On the edges, we consider reaction-diffusion equations, at the vertices, exchanges equations (such as the Kirchoff laws), serve as boundary conditions. Such models are widely studied in the context of Hamilton-Jacobi equations and Schrödinger equations. Note that, in system \eqref{fr gen 1}, the junction is a straight line, and we have two domains connected through this junctions.  Of interest is also the study of more realistic road-field models, with several roads.

Another extension would be to consider ``volume-surface" reaction-diffusion systems. These systems appear in crystallography, to describe the growth of crystals \cite{SV3}, and in cell-biology, to describe the dynamic of proteins within the cell-body and on the cell cortex, or to describe asymmetric stem cell divisions \cite{VS1, SV2}. These models are similar to the road-field models presented here, the main difference being that they involve bounded domains (representing the bulk of a cristal or the cell-body), whose boundaries play the role of our road.

An open problem is to construct a more general formulation of elliptic operators defined on manifold with boundaries. It would be interesting to have a general framework for the Harnack inequalities and existence of eigenfunctions of which all these results are particular cases.

\section{Hypotheses and statement of the results}\label{Our model}
\subsection{Hypotheses and definitions}

We consider in this paper the following linear system:
\begin{equation}\label{lin}
\left\{
\begin{array}{rll}
-Du^{\prime\prime} -c u^{\prime} &= \nu_1 v_1\vert_{y=0}  + \nu_2 v_2\vert_{y=0}   -(\mu_1 + \mu_2) u,  \quad &x\in \mathbb{R}, \\
-d_i\Delta v - c_i\partial_{x}v_i &= a_i(x,y)v_i,  \quad &(x,y)\in \mathbb{R}\times \mathbb{R}^{+}, \\
-d_i\partial_{y}v_i\vert_{y=0} &= \mu_i u-\nu_i v_i\vert_{y=0}, \quad &x\in \mathbb{R}.
\end{array}
\right.
\end{equation}
The parameters $D, d_1, d_2, \nu_1, \nu_2, \mu_1, \mu_2$ are positive constants, while $c, c_1, c_2$ are real numbers, possibly equal to zero, and the $a_i$ are locally Lipschitz continuous functions globally bounded on $\R\times[0,+\infty)$. Observe that this system is more general than~\eqref{lin 1} because the drift terms need not be the same.
%

For notational simplicity, we introduce the following five linear operators (the index $i$ takes the values $1$ and~$2$):
\begin{equation*}
\left\{
\begin{array}{lllrl}
L_{0}(u,v_1,v_2) &:= Du^{\prime\prime}+cu^{\prime} +\nu_1 v_1\vert_{y=0} + \nu_2 v_2\vert_{y=0}- (\mu_1 + \mu_2) u, \\
L_{i}(v) &:= d_i\Delta v+c_i\partial_{x}v+a_i(x,y)v , \\
B_i(u,v) &:= d_i\partial_{y}v\vert_{y=0}+\mu_i u-\nu_i v\vert_{y=0} .
\end{array}
\right.
\end{equation*}

These operators are understood to act on functions $(u,v_1,v_2) \in W^{2,p}_{loc}(\R)\times ( W^{2,p}_{loc}(\R \times [ 0, +\infty)) )^2$. 
For technical reasons, we restrict to $p>2$: first, in order to apply the $L^p$ theory for elliptic equations (see \cite[Chapter 9]{GT}) on a domain $\O \subset \R^2$, one should work in $W^{2,p}(\O)$ for $p\geq 2$. Moreover, by taking $p>2$, the Morrey inequality (see \cite[Theorem 7.17]{GT}) ensures that, if $\O \subset \R^{2}$ is smooth and bounded, then $W^{2,p}(\O) \subset C^{1}(\O)$, and this inclusion is compact.

\medskip
Here is our definition of the generalized principal eigenvalue of \eqref{lin}:
\begin{multline}\label{fgpe}
\lambda_{1}:=\sup\Big\{\lambda \in \mathbb{R} \ : \ \exists (u,v_1,v_2) \geq 0,\ (u,v_1,v_2)\not\equiv (0,0,0) \ \text{ such that } \ \\ 
L_{0}(u,v_1,v_2)+\lambda u \leq 0,\  L_{i}(v_i)+\lambda v_i  \leq 0 ,\  B_i(u,v_i) \leq 0\ \text{ for } \ i=1,2 \Big\}.
\end{multline}
The functions $(u,v_1,v_2)$ in the above and subsequent definitions always belong to $W^{2,p}_{loc}(\R)\times (W^{2,p}_{loc}(\R \times [0,+\infty)))^2$, for some $p>2$. 
Throughout the paper, unless otherwise specified, differential equalities and inequalities are understood to hold almost everywhere.

\subsection{Statement of the results}

We now state our results in the general context of system \eqref{lin}.

The first one is the existence of a generalized principal eigenfunction.
\begin{theorem}\label{thm:gpef}
	There exists a triplet $(u,v_1,v_2)$ 
	of positive functions satisfying, for $i=1,2$,
	\begin{equation*}
	\left\{
	\begin{array}{rll}
	-Du^{\prime\prime} - cu^{\prime}-\nu_1 v_1\vert_{y=0}- \nu_2 v_2\vert_{y=0}+(\mu_1 + \mu_2) u
	&=  \lambda_{1}u, \quad &x \in  \mathbb{R}, \\
	-d_i\Delta v_i-c_i\partial_{x}v_i - a_i v_i  &=  \lambda_{1}v_i, \quad &(x,y) \in   \mathbb{R}\times \mathbb{R}^{+},  \\
	-d_i\partial_{y}v_i\vert_{y=0} - \mu_i u+\nu_i v_i\vert_{y=0} &= 0, \quad &x \in  \mathbb{R}.
	\end{array}
	\right.
	\end{equation*}
	We call $(u,v_1,v_2)$ a {\em generalized principal eigenfunctions} triplet associated with $\l$.
\end{theorem}

The strategy for proving \thm{gpef} consists in deriving a constructive characterization of $\l$ as the limit
of ``standard'' principal eigenvalues, in the spirit of~\cite{BNV,BR2}.
For this, we approximate the half-plane and its boundary by the sequence of truncated domains
$(\O_R)_{R>0}$, $(I_R)_{R>0}$ defined by
$$
\O_R := B_R \cap (\R \times \R^+), \qquad   I_R := (-R,R),
$$
where $B_R$ denotes the (open) ball of radius $R$ and of center $(0,0)$ in $\R^2$. 
Then, we consider the eigenproblem in such domains imposing a Dirichlet condition on the new portion of the boundary 
introduced by the truncation, that is, 
\begin{equation}\label{eigborn}
\left\{
\begin{array}{rll}
-L_{0}(u,v_1,v_2) &= \lambda u \  &\text{ in } I_R, \\
-L_{i}(v_i) &= \lambda v_i \ &\text{ in }\  \Omega_{R}, \\
B_i(u,v_i) &= 0 \ &\text{ in }  \ I_R, \\
v_i &= 0 \ &\text{ on} \   (\partial \Omega_{R})\backslash (I_R\times\{0\}), \\
u(-R)=u(R)&=0 .
\end{array}
\right.
\end{equation}
%
The boundedness of the domain allows us to apply the Krein-Rutman theory which provides the existence of a 
generalized principal eigenvalue, that is, 
a unique $\lambda=\l^R$ such that \eqref{eigborn} admits a solution with $u,v_1,v_2$ positive inside their domains of definition.
Moreover, such solution is unique (up to a scalar multiple); we call it the principal eigenfunctions triplet for \eqref{eigborn}
and we denote it by $(u_R, v_{1,R}, v_{2,R})$.
The Krein-Rutman theorem requires a strong positivity property for the system, that we derive in the Appendix. 
Let us point out that the domain $\O_R$ is non-smooth. From one hand, 
this introduces some technical difficulties, that we overcome using a reflection argument.
From the other hand, the lack of regularity of $\O_R$ is precisely what entails that the eigenfunctions
are smooth up to the boundary. 
Note indeed that, if the domain $\O_R$ were smooth, then at the junction between its boundary 
and the line $\R\times\{0\}$, 
a $C^1$ solution should satisfy both the Dirichlet and the Robin-type condition, whence 
$v_i = -\partial_{y}v_i =0 $. The Hopf Lemma would then yield that $v_i \equiv 0$, for $i=1,2$, and then $u \equiv 0$.

%

Once the principal eigenvalues in the truncated domains are at hand, we derive the convergence result.
\begin{theorem}\label{thm:truncated}
	The generalized principal eigenvalues $(\l^R)_{R>0}$ of \eqref{eigborn} satisfy
	$$
	\lambda_{1}^{R} \underset{R \nearrow +\infty}{\searrow} \lambda_{1}.
	$$
\end{theorem}
Theorems \ref{thm:gpef},\ref{thm:truncated} are proved simultaneously.
The key tool that allows us to pass to the limit is
a Harnack inequality especially devised for road-field models. Here is the precise statement.
\begin{theorem}\label{Harnack}
Let $f\in L^\infty(\R)$ and $g_1 , g_2\in L^\infty(\R\times\R^+)$.
For any $R>0$, there is a positive constant $C$ such that, for any non-negative solution $(u,v_1, v_2) \in W^{2,p}_{loc}(\R)\times (W_{loc}^{2,p}(\R \times [0,+\infty)))^2$ of 
\begin{equation}\label{lin2}
\left\{
\begin{array}{rlll}
-Du^{\prime \prime} -c u^{\prime } -f(x) u&= \sum_i \nu_i \rest{v_i}{y=0}-  (\mu_1 + \mu_2) u, \quad &x \in \mathbb{R}, \\
-d_i\Delta v_i - c_i \partial_{x}v_i - g_i(x,y)v_i&= 0, \quad &(x,y) \in \mathbb{R}\times \mathbb{R}^{+}, \\
-d_i\partial_{y}\rest{v_i}{y=0} &= \mu_i u- \nu_i \rest{v_i}{y=0}, \quad &x \in \mathbb{R},
\end{array}
\right.
\end{equation}
there holds:
$$
\max\left\{ \sup_{I_R }u   ,   \sup_{\Omega_{R}}v_1,  \sup_{\Omega_{R}}v_2 \right\} \leq 
C\min\left\{\inf_{I_{R}}u  ,  \inf_{\Omega_{R}}v_1,  \inf_{\Omega_{R}}v_2 \right\}.
$$
\end{theorem}
Observe that, though our system is, in some sense, ``cooperative", the non-standard coupling through the boundary conditions makes the Harnack inequalities for cooperative systems inapplicable here. 

The Harnack inequality also allows us to 
derive some properties of the generalized principal eigenvalue that will be fundamental in some applications, concerning the continuity and monotonicity with respect to the coefficients. As we already mentioned, the sign of $\l$ characterizes the long-time behavior of solutions of the associated semilinear evolution problem. Therefore, knowing the behavior of $\l$ as a function of the parameters yields qualitative results and allows for a discussion of persistence/extinction in terms of the various parameters, see \cite{BDR2}.

We start by a result on the dependence of the generalized principal eigenvalue with respect to the parameters.

\begin{prop}\label{continuity}
The generalized principal eigenvalue $\l$ for system~\eqref{lin} is a locally Lipschitz-continuous function of the parameters $c, c_i\in\R$, $D,d_i\in\R^+$ and of the functions $a_i\in W^{1,\infty}(\R\times \R^+)$ endowed with the
	$L^\infty (\R \times \R^+)$ norm.
\end{prop}
%
%
%
%
%
We mention that, in the setting we consider, even if the heterogeneous coefficients were smooth, the Kato regularity 
theory for eigenvalues, c.f.~\cite{Kato}, does not apply because the domains we consider are unbounded and because of the unusual coupling.

Next, we derive some monotonicity properties for the generalized principal eigenvalue in the case $c=c_i=0$, which are strict if $\lambda_{1}$ satisfies the following conditions:
\begin{equation}\label{condition strict}
\lim_{R \to +\infty}\bigg( \sup_{ \substack{(x,y) \in \ \R\times\R^+ \setminus \Omega_{R} \\ i \in \{ 1,2\} }} a_i(x,y)\bigg) <- \l \quad \text{ and }\quad \lambda_{1} \leq 0.
\end{equation}
%
%

\begin{prop}\label{monotony}
Assume that $c=c_i=0$. Then the generalized principal eigenvalue~$\l$ for the system \eqref{lin}
is a non-decreasing function of the parameters $d_i$ and $D$ and a non-increasing function of the couple $(a_1,a_2) \in \left(W^{1,\infty}(\R\times \R^+)\right)^2$, endowed with the standard product order.
%
%

Moreover, these monotonicities are
strict if condition \eqref{condition strict} holds.

\end{prop}
The above monotonicity properties in the large sense will be readily derived from a variational formula for $\l^R$, 
of the Rayleigh-Ritz type which we state in Proposition~\ref{variational} below. The proof of the strict monotonicity is much more involved.
 This property is useful in some applications, see \cite{BDR2}, 
 and, as far as we are aware, it is new even for a single~equation.

We point out that the generalized principal 
eigenvalue is not, in general, strictly monotonous with respect to the coefficients, 
even in the case of a single equation, due to the unboundedness of the domain.
As an example, consider the operator
$$
L u := \Delta u + a(x)u\quad\text{in }\R^N.
$$
If $a\equiv 0$ then the generalized principal eigenvalue of $-L$ is $0$. 
However, if $a\leq 0$ is such that $a(x) \to 0$ as $x$ goes to $\pm \infty$, it can be 
easily seen that the generalized principal eigenvalue of $-L$ is still $0$. Observe that condition \eqref{condition strict} is not fulfilled in this example. The validity of the last statement of Proposition \ref{monotony}
relies indeed on the fact that~\eqref{condition strict} entails the exponential decay of the generalized \pf s triple. 
Owing to the Rayleigh-Ritz type formula, this allows us to derive the strict monotonicity.
 
Lastly, it is necessary to assume that $c=c_i=0$, otherwise it might happen that
the \pe\ is not monotonous with respect to the parameters, as exhibited in \cite{BDR2}. 

\medskip
The paper is organized as follows. In Section \ref{SectionHarnack}, we prove the Harnack inequality for system \eqref{lin}, Theorem \ref{Harnack}. It will be the keystone for the next Section \ref{SectionLinearized}, where we study properties of the generalized principal eigenvalue $\l$. We start with showing that the definition \eqref{fgpe} of $\l$ does provide a real number, satisfying some explicit bounds.
 Next, in Section \ref{section omega}, we focus on the case of the bounded domains $\O_R$. In particular, we provide a variational characterization of $\l^R$, similar to the classical Rayleigh-Ritz formula. Then, we devote Section \ref{cv vp} to proving Theorems 
 \ref{thm:gpef} and~\ref{thm:truncated}. 
 We will follow a standard strategy consisting in constructing a triplet $(u,v_1,v_2)$ that achieves the maximum in the formula \eqref{fgpe}. This triplet will be obtained as the limit of a subsequence of the principal eigenfunctions
$((u_R,v_{1,R},v_{2,R}))_{R>0}$ of \eqref{eigborn}, suitably normalized. We then prove Propositions \ref{continuity} and \ref{monotony} in Section \ref{m et c}.

%

\section{The Harnack inequality}\label{SectionHarnack}

This section is dedicated to proving the Harnack inequality, Theorem \ref{Harnack}. We start with a weaker version.
\begin{prop}\label{Prop Harnack}
For any $R>0$, there is a positive constant $C$ such that, for any non-negative solution $(u,v_1,v_2)$ of \eqref{lin2}, there holds
$$
\max\left\{  \sup_{I_{R} }u ,\  \sup_{\Omega_{R}}v_1 ,\ \sup_{\Omega_{R}}v_2 \right\} \leq 
C\max\left\{  \inf_{I_{R}}u  , \  \inf_{\Omega_{R}}v_1 , \  \inf_{\Omega_{R}}v_2   \right\}.
$$
\end{prop}

To prove this result, we follow a standard approach (see \cite[Theorem 8.20]{GT} for instance) consisting in deriving separately
a weak Harnack inequality for supersolutions and a local maximum principle for subsolutions. We will make use of the Moser iteration method, see \cite{M}. We say that a triples $(u,v_1,v_2)$ is a supersolution (respectively a subsolution) of \eqref{lin2} if it satisfies the system with the signs $=$ replaced by~$\geq$ (respectively $\leq$).
\begin{lemma}\label{WH}
For any $R>0$ and $p\geq1$, there is a positive constant $C$ such that, for any non-negative supersolution $(u,v_1,v_2)$ of \eqref{lin2}, there holds
$$
\| u\|_{L^{p}(I_{R+1})} \leq C \inf_{I_R }u \ \text{ and } \
\| v_i \|_{L^{p}( \Omega_{R+1})} \leq C\inf_{ \Omega_{R}}v_i , \ i=1,2.
$$
\end{lemma}

\begin{lemma}\label{LMP}
For any $R>0$ and $p>1$, there is a positive constant $C$ such that, for any non-negative subsolution $(u,v_1,v_2)$ of \eqref{lin2}, there holds
$$
\max\left\{ \sup_{I_R}u, \  \sup_{ \Omega_{R}}v_1  ,\ \sup_{ \Omega_{R}}v_2 \right\} 
\leq C \max\{  \| u\|_{L^{p}( I_{R+1})}, \  \| v_1 \|_{L^{p}( \Omega_{R+1})}, \ \| v_2 \|_{L^{p}( \Omega_{R+1})}   \}.
$$
\end{lemma}

We start by proving Lemma \ref{WH}. 
The main observation is that the functions $u$, $v_1$ and $v_2$ are supersolutions of some elliptic equations 
that do not depend on the other functions.
This allows us to apply the usual weak Harnack inequality, which is valid for supersolutions. 
This works directly with $u$, whereas for $v_1$, $v_2$ we will need to extend $v_1, v_2$ to the whole space by oblique reflection.

\begin{proof}[Proof of Lemma \ref{WH}.]
	Let $R>0$ and $p\geq1$ be given, and let $(u,v_1,v_2)$ be a non-negative supersolution of \eqref{lin2}. 

First, observe that the function $u$ is a supersolution of an elliptic equation:
$$-Du^{\prime \prime} -c u^{\prime } -(f(x)-\mu_1 - \mu_2) u\geq 0.$$ 
Hence, we can apply to $u$ the usual weak Harnack inequality (see \cite[Theorem 8.18]{GT})
to get
$$\| u\|_{L^{p}(I_{R+1})} \leq C \inf_{I_R}u,$$
where $C$ depends on $p$ and $R$.

Let us focus on $v_1$ (the situation for $v_2$ is analogous). 
In order to apply the (interior) weak Harnack inequality to $v_1$,
we want  to extend it to a supersolution of an elliptic elliptic equation on the whole plane. To do so, we start by defining
$$
w(x,y) := v_1(x,y) e^{-\frac{\nu_1}{d_1}y}.
$$
The boundary condition yields that $-d_1\partial_{y}w\vert_{y=0} \geq \mu_1 u \geq 0 $, and therefore
\begin{equation*}
\left\{
\begin{array}{rlll}
-d_1\Delta w-2\nu_1 \partial_{y}w -c_1 \partial_{x}w-(\frac{\nu_1^2}{d_1}+g_1)w &\geq  0 \quad &\text{on}\  \R \times \R^+, \\
-d_1\partial_{y}w\vert_{y=0} &\geq 0 \quad &\text{on} \R.
\end{array}
\right.
\end{equation*}
We extend $w$ to $\R^2$ by symmetry, by setting $\t w(x,y) := w(x,\vert y \vert)$, for $(x,y) \in \R^2$. Consider now the operator $\t L$ defined by 
$$\tilde{L}:=d_1\Delta + 2\nu_1 H(y) \partial_{y} +c_1\partial_{x} +\frac{\nu_1^2}{d_1}+g_1,$$
where $H(y) = 1$ if $y\geq0$ and $H(y)=-1$ if $y<0$. Because~$\partial_{y}w\vert_{y=0}\leq0$, it is readily seen that~$\t w$ 
satisfies $-\tilde{L}\t w \geq 0$ in $\O_R$, in the weak (i.e., $W^{1,2}$) sense. Then~$\t w$ is supersolution of an elliptic equation on $\R^2$.
%
%
%
%
%
%
We can therefore apply the weak Harnack inequality to $\t w$ to get
$$
\| \t w\|_{L^{p}(\Omega_{R+1})} \leq C \inf_{\Omega_{R}} \t w,
$$
for some positive constant $C$ depending on $R,p$. Recalling the definition of $\t w$, we get the weak Harnack inequality for $v_1$.
\end{proof}

We now turn to the proof of Lemma \ref{LMP}. We will apply the same strategy as for the proof of the usual local maximum principle for a single equation (see \cite[Theorem 8.17]{GT}): the idea is to write the system in a weak form and apply it to a suitable pair of test functions. Because we are dealing with a system and not with a single equation, a coupling term will appear.

\begin{proof}[Proof of Lemma \ref{LMP}] 
Let $R>0$ and $p\geq1$ be given, and let $(u,v_1,v_2)$ be a 
nonnegatice subsolution of \eqref{lin2}. We define 
$$
\tilde{u} := (\mu_1 + \mu_2) u \ \text{ and } \ \tilde{v}_i := \left( \frac{\mu_1 + \mu_2}{\mu_i} \right) \nu_{i} v_i,
$$
These are subsolutions of
\begin{equation*}
\left\{
\begin{array}{rlll}
-\t D u^{\prime \prime} -\t c u^{\prime } -\t f(x) u&=  \sum_i \frac{\mu_i}{\mu_1 + \mu_2}(\rest{v_i}{y=0}-  u), \quad &x \in \mathbb{R}, \\
-\t d_i\Delta v_i - \t c_i \partial_{x}v_i - \t g_i(x,y)v_i&= 0, \quad &(x,y) \in \mathbb{R}\times \mathbb{R}^{+}, \\
-\t d_i\partial_{y}\rest{v_i}{y=0} &= \frac{\mu_i}{\mu_1 + \mu_2} (u-  \rest{v_i}{y=0}), \quad &x \in \mathbb{R},
\end{array}
\right.
\end{equation*}
for some coefficients $\t D, \t d_i, \t c, \t c_i>0$ and functions $\t f,\t g_i$.
In the rest of the~proof, 
we omit the tilde on the above terms as well as on the subsolutions $\t u, \t v_1, \t v_2$.
We start by writing the system in a weak form. In order to do so, 
we take $(\phi, \psi_1,\psi_2)$ regular (${\rm C}^{1}$) non-negative test functions such that $\phi$ has compact support in $I_{R+1}\times \{0\}$ and $\psi_1$ and $\psi_2$ have compact supports in $\ol{\O}_{R+1}$ and we multiply the first equation, i.e., the equation on the road, by $\phi$ and the $i$-th equation (the one in the $i$-th side of the field) by $\psi_i$ ($i=1,2$) and integrate by parts. Using the boundaries equations, we get:
\begin{equation*}
\begin{array}{rl}
\int_{I_{R+1}} \Big( Du^{\prime}\phi' -cu'\phi -f(x)u\phi \Big) &\leq  \sum_i \frac{\mu_i}{\mu_1 + \mu_2} \int_{I_{R+1}}\Big( v_i\vert_{y=0}-u\Big)\phi,
\end{array}
\end{equation*}
and
$$
\int_{\Omega_{R}} \Big( d_i\nabla v_i\. \nabla \psi_i -c_i\partial_{x}v_i \psi -g_i(x,y)v_i\psi_i \Big) \leq
\frac{\mu_i}{\mu_1 + \mu_2}  \int_{I_{R+1}}\left(u-  \rest{v_i}{y=0} \right)\psi_i\vert_{y=0}.
$$
Summing all the inequalities leads to
\begin{equation}\label{variational2}
\begin{array}{rr}
\int_{I_{R+1}} \Big( Du^{\prime}\phi' -cu'\phi -f(x)u  \Big)
\quad + &\sum_i\int_{\Omega_{R+1}} \Big( d_i\nabla v_i\. \nabla \psi_i -c_i\partial_{x}v_i \psi_i -g_i(x,y)v_i\psi_i  \Big) \\
& \leq \sum_i \frac{\mu_i}{\mu_1 + \mu_2} \int_{I_{R+1}}\left( u-  \rest{v_i}{y=0} \right)(\psi_i\vert_{y=0}-\phi).
\end{array}
\end{equation}
Now, we consider $R_{1}, R_{2} \in \mathbb{R}$ to be specified later such that $R\leq R_{1}<R_{2} \leq R+1$. For $\beta>0$ set
$$
\psi_i(x,y):= v_i^{\beta}(x,y)\eta^{2}(x,y) \ \text{ and } \ \phi(x):=u^{\beta}(x)\eta^{2}(x,0),$$
 where $\eta \ : \ \R \times [0,+\infty) \to \R$ is a non-negative smooth truncation function such that $\eta=1$ on $\Omega_{R_{1}}$, $\eta = 0$ on $\R \times [0,+\infty) \backslash\Omega_{R_{2}}$ and $\eta \leq 1$. Moreover, we choose it such that $\vert \nabla \eta \vert \leq \frac{2}{R_{2}-R_{1}}$. Note that $\phi, \psi_i$ might not be $C^1$ if $\beta <1$. To avoid this diffuculty we assume that $u$, $v_i$ are strictly positive (this is automatically verified if we are dealing with non-zero solutions). In the general case, it is sufficient to consider $(\overline{u},\overline{v}_{1}, \overline{v}_{2}):=(u+k,v_1+\frac{\mu_1}{\nu_1}k,v_2+\frac{\mu_2}{\nu_2}k)$, with $k>0$, and to let $k$ go to zero, see \cite[Theorem 8.17]{GT} for the details.

The crucial observation is that, since $\beta>0$, we have
$$
\int_{I_{R+1}}\left(u -  v_i\vert_{y=0}\right)(\psi_i\vert_{y=0} - \phi)= \int_{I_{R+1}}\left( u -  v_i\vert_{y=0}\right)( v_i\vert_{y=0}^{\beta} -   u^{\beta})\eta\vert_{y=0}^{2}\leq 0,
$$ 
Implementing these ``test functions" $\phi, \psi_1, \psi_2$ in \eqref{variational2}, we get

\small
\[
\begin{split}
\sum_i&\left(\beta d_i\int_{\Omega_{R+1}} \vert \nabla v_i \vert^{2}v_i^{\beta-1}\eta^{2} +
2d_i \int_{\Omega_{R+1}}(\nabla v_i\. \nabla \eta) v_i^{\beta} \eta - \int_{\Omega_{R+1}}c_i\partial_{x}v_i v_i^{\beta}\eta^{2}  - \int_{\Omega_{R+1}}g_iv_i^{\beta+1}\eta^{2}\right) \\ 
&+ \beta D \int_{I_{R+1}} \vert u^{\prime} \vert^{2}u^{\beta-1}\eta^{2} +2D\int_{I_{R+1}}u^{\prime}  \eta^{\prime} u^{\beta} \eta -\int_{I_{R+1}}cu' u^\beta \eta^2  -\int_{I_{R+1}}f u^{\beta+1} \eta^{2}
 \leq 0.
\end{split}
\]
\normalsize
Now, for any $\e >0$, we estimate two of the above terms as follows:
$$
\left|  2\int_{\Omega_{R+1}} \Big( \nabla v\. \nabla \eta \Big) v^{\beta} \eta \right| \leq 
\int_{\Omega_{R+1}} \Big(\e \eta^{2} \vert \nabla v\vert^{2} v^{\beta-1}  +\frac{1}{\e}\vert \nabla \eta \vert^{2}  v^{\beta+1} \Big),
$$
$$
\left|2\int_{I_{R+1}}u^{\prime}  \eta^{\prime} u^{\beta} \eta\right| \leq
\int_{I_{R+1}} \Big(\e \eta^{2} |u'|^{2}u^{\beta-1} +\frac{1}{\e}|\eta'|^{2}u^{\beta+1} \Big).
$$
Two other terms are estimated by integration by part:
$$
\left|  \int_{\Omega_{R+1}} \partial_{x} v \eta^{2} v^{\beta}  \right| \leq 
\int_{\Omega_{R+1}} \Big(\vert \nabla \eta\vert^{2} + \eta^{2}\Big)v^{\beta+1},
$$
$$
\left| \int_{I_{R+1}}u^{\prime}  \eta^{2} u^{\beta} \right|\leq
\int_{I_{R+1} }\Big(\vert \eta^{\prime}\vert^{2} + \eta^{2} \Big)u^{\beta+1}.
$$
Using these inequalities together with $\left|  \int_{\Omega_{R+1}}gv^{\beta+1}\eta^{2} \right| \leq \|g \|_{L^{\infty}} \int_{\Omega_{R+1}}v^{\beta+1}\eta^{2}$ and $\left|  \int_{I_{R+1}}fu^{\beta+1}\eta^{2} \right| \leq \|f \|_{L^{\infty}} \int_{I_{R+1}}u^{\beta+1}\eta^{2}$,
yields:

\[
\begin{split}
 \sum_i \left(  d_i(\beta -\e)\int_{\Omega_{R+1}}\vert \nabla v_i \vert^{2}v_i^{\beta-1}\eta^{2} \right)+
D(\beta-\e)\int_{I_{R+1}}\vert u^{\prime} \vert^{2}u^{\beta-1}\eta^{2}\quad\quad\quad\quad\quad\quad\quad\quad \\
\leq \sum_i \left(  \frac{d_i}{\e} \int_{\Omega_{R+1}}\vert \nabla \eta\vert^{2}v_i^{\beta+1}\right)+\frac{D}{\e} \int_{I_{R+1} }\vert \eta^{\prime} \vert^{2}u^{\beta+1} +\sum_i   \| g_i \|_{L^{\infty}} \int_{\Omega_{R+1}}v_i^{\beta+1}\eta^{2} \\+ \| f\|_{L^{\infty}}\int_{I_{R+1}}u^{\beta+1}\eta^{2} 
+\sum_i  \vert c_i \vert\int_{\Omega_{R+1}}\Big(\vert \nabla \eta\vert^{2} + \eta^{2}\Big)v_i^{\beta+1} + \vert c \vert \int_{I_{R+1}}\Big(\vert \eta^{\prime}\vert^{2} + \eta^{2}\Big)u^{\beta+1}.
\end{split}
\]
We now define 
$$
w_{i}:=v_i^{\frac{\beta+1}{2}} \ \text{ and } \ z:=u^{\frac{\beta+1}{2}},
$$
and let $\e \leq \min\{ d_1, d_2, D\}$. Then, there is a constant $C>0$ depending only on $\| g_i \|_{L^{\infty}}$, $\| f \|_{L^{\infty}}$, $\vert c_i \vert$ and $\vert c \vert$, such that the above inequality yields:

\[\begin{split}
&\sum_i \left(  4d_i \frac{\beta -\e}{(\beta+1)^{2}}\int_{\Omega_{R+1}}\vert \nabla w_{i} \vert^{2}\eta^{2}\right) +
4D \frac{\beta -\e}{(\beta+1)^{2}}\int_{I_{R+1}}\vert z^{\prime} \vert^{2}\eta^{2} \\
& \leq  C \sum_i  \left( \frac{d_i}{\e} \int_{\Omega_{R+1}}(\vert \nabla \eta\vert^{2} +\eta^{2})w_{i}^{2}  \right)      
+\frac{D}{\e} \int_{I_{R+1}} (\vert \eta^{\prime} \vert^{2}+ \eta^{2})z^{2}.
\end{split}\]
Here and in the sequel, $C$ stands for a generic constant which only depends on the coefficients
of the equation. If we let $\e = \min\{ \frac{\beta}{2} , D,d_1, d_2\}$, we get

\begin{equation}\label{vartest4}
\begin{array}{lll}
\sum_i&\int_{\Omega_{R+1}}\vert \nabla w_{i} \vert^{2}\eta^{2} +\int_{I_{R+1}}\vert z^{\prime} \vert^{2}\eta^{2}
\leq \\ 
&C\frac{(\beta+1)^2}{\beta}\left(\sum_i   \int_{\Omega_{R+1}}(\vert \nabla \eta\vert^{2} +\eta^{2})w_{i}^{2} + \int_{I_{R+1}} (\vert \eta^{\prime} \vert^{2}  +\eta^{2})z^{2}\right).
\end{array}
\end{equation}
Now, we apply the Sobolev inequalities (see \cite[Theorem 7.26]{GT}) for, say, the space $L^4$ (remember that we are in space dimension $1$ on the road and $2$ in the field and that $\eta \in H^{1}_{0}(\Omega_{R+1}\cup (I_{R+1}\times \{ 0\})$), we thus get

$$\| \eta w_{i}\|_{L^4(\O_{R+1})} \leq C \| \nabla (\eta w_{i} )\|_{L^2(\O_{R+1})} \leq C \|\eta( \nabla w_{i} )\|_{L^2(\O_{R+1})} +C \| (\nabla \eta) w_{i} \|_{L^2(\O_{R+1})}$$ and 
$$\| \eta z \|_{L^4(I_{R+1})} \leq C \|  (\eta z )^{\prime}\|_{L^2(I_{R+1})} \leq C \|  \eta z^{\prime}\|_{L^2(I_{R+1})}+ C \|  \eta^{\prime} z \|_{L^2(I_{R+1})}.$$ 
%
%
Combining this with \eqref{vartest4} we get
\begin{equation*}
\sum_i \| \eta w_{i}\|_{L^4}^{2} +\| \eta z\|_{L^4}^{2}
\leq C\frac{(\beta+1)^2}{\beta}\left(\sum_i \int_{\Omega_{R+1}}\Big(\vert \nabla \eta\vert^{2} +\eta^{2}\Big)w_{i}^{2}     
 + \int_{I_{R+1}} \Big(\vert \eta^{\prime} \vert^{2}  +\eta^{2}\Big)z^{2}\right),
\end{equation*}
Where $C$ is a positive constant depending now on $R$ (but not on $\beta$).
Because $\eta=1$ on $\O_{R_1}$ and
$ \vert \nabla \eta \vert \leq \frac{2}{R_{2}-R_{1}}$, we derive
\begin{equation*}
\sum_i\| w_{i}\|_{L^4(\Omega_{R_{1}})}^{2} +\|  w_{2}\|_{L^4(I_{R_1})}^{2} 
\leq C\frac{(\beta+1)^2}{\beta(R_{2}-R_{1})^{2}}\left(\sum_i \|w_{i}\|_{L^{2}(\Omega_{R_{2}})}^{2}    + \|z\|_{L^{2}(I_{R_2})}^{2} \right) .
\end{equation*}
Writing for short $\gamma:=\beta+1$, 
and remembering that $w_{i}=v_i^{\frac{\gamma}{2}}$ and $z=u^{\frac{\gamma}{2}}$, 
this inequality rewrites as
\begin{equation}\label{M}
\sum_i \|v_i\|_{L^{2\gamma}(\O_{R_1})}^\gamma +\|u\|_{L^{2\gamma}(\O_{R_1})}^\gamma  
\leq C\frac{\gamma^2}{(\gamma-1)(R_{2}-R_{1})^{2}}
\left(\sum_i \|v_i\|_{L^{\gamma}(\O_{R_2})}^\gamma +\|u\|_{L^{\gamma}(I_{R_2})}^\gamma  \right).
\end{equation}
Then, for $q,r >0$, define
$$
M(q,r) := \max \{\|u\|_{L^q(I_r)},  \|v_1\|_{L^q(\O_r)},  \|v_2\|_{L^q(\O_r)} \}.
$$
Keeping in mind that $\gamma>1$ and $R \leq R_1<R_2\leq R+1$, \eqref{M} yields
\begin{equation*}
\big( M(2\gamma, R_1)\big)^{\gamma} 
\leq 3C\frac{\gamma^2}{(\gamma-1)(R_{2}-R_{1})^{2}}\big(M(\gamma , R_2)\big)^{\gamma}.
\end{equation*}
We recursively apply this inequality, starting from 
an arbitrary $\gamma=p>1$, $R_1=R$, $R_2=R+1/2$ and taking, at the $m$-th step,
$$\gamma=2^mp,\quad R_1=R+\sum_{n=1}^m 2^{-n}, \quad R_2=R+\sum_{n=1}^{m+1} 2^{-n}.$$ 
Going to the limit we finally get
$$M(+\infty, R)\leq \prod_{m=0}^\infty \left(3C\frac{2^{2m}p^{2}}{(2^{m}p-1)2^{-2m-2}}\right)^{\frac{1}{2^mp}}
M(p,R+1),$$
where 
\begin{multline*}
M(+\infty,r) = \lim_{q \to +\infty} \max\{\|u\|_{L^{q}(I_r)} , \|v_1\|_{L^{q}(\O_r)}, \|v_2\|_{L^{q}(\O_r)} \} \\ =  \max \{\|u\|_{L^{\infty}(I_r)}, \|v_1\|_{L^{\infty}(\O_r)}, \|v_2\|_{L^{\infty}(\O_r)} \}.
\end{multline*}
Observing that the infinite product in the above inequality converges, our result follows.
\end{proof}

Lemmas \ref{WH} and \ref{LMP} yield Proposition \ref{Prop Harnack}. In order to prove Theorem \ref{Harnack}, we derive the following.

\begin{prop}\label{min max}
For any $R>0$, there is a constant $C>0$ such that, for any non-negative solution $(u,v_1,v_2)$ of \eqref{lin2}, there holds, for $i=1,2$,
$$
 \frac{1}{C} \sup_{I_R}u \leq \inf_{\O_R} v_i \leq \sup_{\O_R} v_i  \leq C \inf_{I_R}u.
$$
\end{prop}

\begin{proof} Let $R>0$ be fixed. We only prove the result for $i=1$. We argue by contradiction and we assume that there is a sequence of non-negative solutions $((u_{n},v_{1,n}, v_{2,n}))_{n\in \mathbb{N}}$ of \eqref{lin2} and two sequences $((x_n , y_n))_{n\in \mathbb{N}}$ and $(z_n )_{n \in \mathbb{N}}$ of points in $\O_{R}$ and $I_R$ respectively such that
$$
\frac{v_{1,n}(x_n , y_n)}{u_{n} (z_n )} \underset{n\to +\infty}{\longrightarrow} +\infty.
$$
Normalize $(u_n , v_{1,n} , v_{2,n})$ so that 
$$
v_{1,n} (x_n , y_n ) + v_{2,n} (x_n , y_n ) + u_n(z_n) = 1.
$$
 Hence, $u_n (z_n) \to 0$ as $n$ goes to $+\infty$. Up to extraction, we now assume that there are $(x_\infty,y_\infty) \in \ol \O_{R}$ and $z_\infty \in \ol I_R$ such that $(x_n , y_n ) \to (x_\infty , y_\infty ) $ and $z_n \to z_\infty$ as $n$ goes to $+\infty$.

Then, for any $R' >R$, using Proposition \ref{Prop Harnack} we find that there is a constant $C>0$ independent of $n$ such that
$$
\max\left\{ \sup_{I_{R'} }u_n,\   \sup_{\Omega_{R'}}v_{1,n}, \  \sup_{\Omega_{R'}}v_{2,n}  \right\} \leq 
C\max\left\{ u_n (z_n), \ v_{1,n}(x_n , y_n), v_{2,n}(x_n , y_n) \right\} \leq C.
$$
Because of this, by the elliptic estimates, $((u_n,v_{1,n}, v_{2,n}))_{n\in \N}$ converges (up to some subsequence) locally uniformly in $\R, \R \times [0, +\infty), \R \times [0, +\infty)$ to a solution $(u_{\infty},v_{1,\infty},v_{2,\infty})$ of \eqref{lin2}, such that $u_{\infty}(z_{\infty}) = 0$ and $v_{1,\infty}(x_\infty , y_\infty )+v_{2,\infty}(x_\infty , y_\infty ) = 1$. The strong maximum principle yields $(u_{\infty}, v_{1,\infty},v_{2,\infty}) \equiv (0,0,0)$, which is the contradiction we sought. We have then shown that
$$
\sup_{\O_R}v_1 \leq C \inf_{I_R} u,
$$
for some constant $C>0$ independent of $u,v_1,v_2$. We can use the same arguments to show that
$$
\sup_{I_R}u \leq C \inf_{\O_R} v_1.
$$
\end{proof}
Finally, combining Proposition \ref{min max} with Proposition \ref{Prop Harnack} we derive Theorem \ref{Harnack}.

\section{The generalized principal eigenvalue} \label{SectionLinearized}

In this section, we prove some properties of the generalized principal eigenvalue $\l$ defined by \eqref{fgpe}.
As a preliminary we show that such quantity is a well-defined real number.
\begin{prop}\label{bound lambda}
The generalized principal eigenvalue $\l$ defined by \eqref{fgpe} is a real number and satisfies
$$
\min\{ 0, -\sup a_1, -\sup a_2 \} \leq \l \leq \min\left\{\frac{c^2}{4D} + \mu_1 + \mu_2, \lambda_D(-L_i)  \right\},
$$
where $\lambda_D (-L_i)$ denotes the generalized principal eigenvalue of the operator $-L_i$ defined on $\R \times \R^+$ with Dirichlet boundary condition.
\end{prop}
For the theory of the generalized principal eigenvalue for elliptic operators in unbounded domains under Dirichlet boundary conditions we refer to \cite{BR1, BR2}. There, it is proved in particular that $\lambda_D(-L_i) \in \R$. 
\begin{proof}
To show the lower bound, we use 
$$
(u,v_1,v_2) = \left(1, \frac{\mu_1}{\nu_1}, \frac{\mu_2}{\nu_2}\right)
$$
 as a ``test function" in \eqref{fgpe} with $\lambda = \min\{ 0, -\sup a_1, -\sup a_2 \}$. The upper estimate that ensures that $\l \in \R$ is readily derived by observing that, if $\lambda \in \R$ and $(u,v_1,v_2)$ satisfy the conditions in \eqref{fgpe}, then
  $$
 D u^{\prime \prime} + c u^{\prime} -(\mu_1 + \mu_2) u + \lambda u \leq 0.
 $$
 By definition of the generalized principal eigenvalue for the scalar operators $-L_0$, which is equal to $\frac{c^2}{4D} + \mu_1 + \mu_2$ (see \cite{BR1, BR2}), this implies that $\lambda \leq \frac{c^2}{4D} + \mu_1 + \mu_2$.
 
 
 Likewise, since 
 $$
 d_i\Delta v_i +c_i \partial_{x}v_i +a_i v_i + \lambda v_i \leq 0.
 $$
It follows from the definition of $\lambda_{D}(-L_i)$ that $\l \leq \lambda_D(-L_i)$, for $i=1,2$.
\end{proof}

\subsection{Principal eigenvalue on truncated domains}\label{section omega}

We now turn to the eigenproblem in the truncated domain.
The existence of the principal eigenvalue is contained in the following.
\begin{prop}\label{propKR}
For $R>0$, there is a unique $\lambda_{1}^{R} \in \R$ and a triplet of non-negative, not
identically equal to zero functions $(u_R,v_{1,R},v_{2,R})$ satisfying \eqref{eigborn}. This eigenfunctions triplet is unique (up to multiplication by a positive scalar). Moreover, any other eigenvalue~$\lambda$ 
of~\eqref{eigborn} satisfies $\Re(\lambda) > \lambda^R_1$.
\end{prop}
This proposition follows from the Krein-Rutman theorem. Because its application is not standard in the present framework, 
due to the coupling condition through the boundary and the lack of regularity of $\O_R$, we provide the details it in the Appendix.

We will apply the Krein-Rutman theorem in the following space:
\begin{multline}\label{space}
\mc C_R := \{ (u,v_1,v_2) \in C^{1}(\ol{I_R})\times (C^{1}(\ol{\O_R}))^2 \ : \ u=0 \text{ on } \partial I_R \\  \text{ and } v_1 = 0, \ v_2 = 0 \text{ on } \partial \O_R \backslash( I_R \times \{ 0 \})   \}.
\end{multline}
A key step in applying the Krein-Rutman theorem, that will also be useful in the sequel, 
is the following strong positivity property for the system \eqref{eigborn}.
%
\begin{lemma}\label{int}
Let $R>0$, $\lambda \in \R$ and $(u,v_1,v_2) \geq 0$, $(u,v_1,v_2) \not\equiv (0,0,0)$ be a supersolution of \eqref{eigborn}. 
Then, for every $(\phi,\psi_1,\psi_2) \in \mc C_R$, there exists $\e>0$ such that
$$
(u,v_1,v_2) \geq \e (\phi,\psi_1,\psi_2).
$$
\end{lemma}
This lemma is proved in the Appendix. It provides the strong maximum principle for our systems. Namely, a non-negative supersolution of \eqref{lin} or of \eqref{eigborn} is either~$(0,0,0)$ or it is positive in $\R$ and $\R \times [0, +\infty)$, or in $I_R$ and $\Omega_R \cup (I_R \times \{ 0\})$ respectively. In particular, the eigenfunctions $(u_R, v_{1,R}, v_{2,R})$ given by Proposition \ref{propKR} are strictly positive in $I_R$ and $\overline\O_R\setminus\partial B_R$ respectively.

Let us show that the definition \eqref{fgpe} of the generalized principal eigenvalue coincides with the principal eigenvalue $\lambda_{1}^R$ when restricting the equations to $I_R$ and~$\O_R$.

\begin{prop}\label{varform}
Let $R>0$. Then, $\l^R$ coincides with the right-hand side of \eqref{fgpe} restricted to $(u,v_1,v_2)$ in $W^{2,p}(I_R)\times (W^{2,p}(\O_R))^2$.

\end{prop}

\begin{proof}
Let $\overline{\lambda}$ denote the right hand side of \eqref{fgpe} restricted to $(u,v_1,v_2)$ in $W^{2,p}(I_R)\times (W^{2,p}(\O_R))^2$. Then, using the principal eigenfunctions $(u_{R},v_{1,R},v_{2,R})$ of \eqref{eigborn} provided by Proposition \ref{propKR} as ``test functions" in \eqref{fgpe}, we readily get 
$$
\overline{\lambda}\geq \lambda_{1}^{R}.
$$
To prove the equality, assume by contradiction that  $\overline{\lambda}>\lambda_{1}^{R}$ and choose $\lambda \in ( \lambda_{1}^{R} ,\overline{\lambda})$ so that we can find $(u,v_1,v_2)\geq (0,0,0)$, $(u,v_1,v_2)\not\equiv (0,0,0)$ such that $L_{0}(u,v_1,v_2)+\lambda u \leq 0$, $L_{i}(v_i)+\lambda v_i \leq 0 $ and $B_i(u,v_i) \leq 0$, for $i=1,2$. In other words, $(u,v_1,v_2)$ is a supersolution of \eqref{eigborn}. Then Lemma \ref{int} implies that there is $\e>0$ such that 
$$
(u,v_1,v_2) \geq \e (u_R,v_{1,R},v_{2,R}).
$$
We let $\e^{\star}$ denote the largest $\e>0$ such that this inequality holds true. We define
$$
(\tilde{u}, \tilde{v}_1,\tilde{v}_2) := (u - \e^{\star}u_{R}, v_1 - \e^{\star}v_{1,R}, v_2 - \e^{\star}v_{2,R}) \geq (0,0,0).
$$
Observe that $(\tilde{u}, \tilde{v}_1,\tilde{v}_2)  \not\equiv (0,0,0)$ because $\lambda>\lambda_{1}^{R}$. Then, $(\tilde{u}, \tilde{v}_1,\tilde{v}_2) $ is still a supersolution of \eqref{eigborn}. Owing to Lemma \ref{int}, we find $\t \e >0$ so that $(\tilde{u}, \tilde{v}_1,\tilde{v}_2) > \t \e  (u_R,v_{1,R},v_{2,R})$, but this would contradict the definition of $\e^{\star}$.
\end{proof}

If $c=c_i=0$ in \eqref{lin} (i.e., there are no drift-terms in the system), we can derive the following alternative formula for $\lambda_{1}^{R}$.

\begin{prop}\label{variational}
Assume that $c=c_i=0$. Then the principal eigenvalue $\lambda_{1}^{R}$ for~\eqref{eigborn} satisfies
\begin{equation}\label{eqvariational}
\lambda_{1}^{R} = \min_{\substack{(u,v_1,v_2) \in \mathcal{H}_{R}\\ (u,v_1,v_2)\not\equiv(0,0,0)}} \frac{    \int_{I_R}D \vert u^{\prime} \vert^{2}  + \sum_{i=1,2}\left( \frac{\nu_i}{\mu_i}\int_{\Omega_{R}} \left( d_i\vert \nabla v_i \vert^{2} - a_i v^{2} \right) +  \frac{1}{\mu_i}\int_{I_R}(\mu_i u- \nu_i v_i\vert_{y=0})^{2} \right)  }{ \int_{I_R}u^{2} +\sum_{i=1,2}  \frac{\nu_i}{\mu_i} \int_{\Omega_{R}}v_i^{2} }
\end{equation}
where \footnote{The space $H^{1}_{0}(\Omega_{R}\cup (I_{R}\times\{0\}))$ is the completion in $H^1(\O_R)$ of functions in $C^{\infty}(\R^2)$ with support in $\Omega_{R}\cup (I_{R}\times\{0\})$. In particular, they do not need to vanish on $I_R \times \{ 0 \}$.}
$$
\mathcal{H}_{R} :=   H^{1}_{0}(I_R) \times \Big( H^{1}_{0}(\Omega_{R}\cup (I_{R}\times\{0\}))\Big)^2.
$$
%

\end{prop}
This result is an adaptation to our model of the classical Rayleigh-Ritz formula, see \cite[Theorem 8.37]{GT} for instance. The usual way of proving the Rayleigh-Ritz formula is to use the spectral theorem, which provides an orthogonal basis. Here, it is more convenient to use a direct approach.
\begin{proof}[Proof of Proposition \ref{variational}]
We set $\beta_i := \frac{\nu_i}{\mu_i}$ and $\gamma_i := \frac{1}{\mu_i}$. Consider the following functional:
$$
\mc E(u,v_1,v_2) :=   \int_{I_R}D \vert u^{\prime} \vert^{2}  + \sum_{i=1,2}\beta_i \int_{\Omega_{R}} \left( d_i \vert \nabla v_i \vert^{2} -a_i v^{2} \right) + \sum_{i=1,2} \gamma_i \int_{I_R}(\mu_i u- \nu_i v_i\vert_{y=0})^{2},
$$
acting over the space $\mc H_R$. We look for minimizers of $\mc E$ over the set
$$
\mc S := \left\{ (u,v_1,v_2) \in \mc H_R \ : \  \int_{I_R}u^{2} + \sum_{i=1,2} \beta_i \int_{\Omega_{R}}v_i^{2} =1 \right\}.
$$

\medskip 
\emph{Step 1. Existence of minimizers.}\\
Consider a minimizing sequence $((u_n,v_{1,n}, v_{2,n}))_{n\in \N}$ of $\mc E$ over $\mc S$. Observe that the $L^2$ norm of functions in $\mathcal{S}$ are uniformly bounded. One then readily checks that the $H^{1}$ norm of $u_n$ and $v_{1,n},v_{2,n}$ are bounded independently of $n$. Therefore, up to extraction, the sequences $(u_n)_{n\in \N}$ and $(v_{1,n})_{n\in \N}$, $(v_{2,n})_{n\in \N}$ converge weakly in $H^{1}$ and strongly in $L^2$ (because $H^1$ is compactly imbedded in $L^2$) to some limits $
u_\infty$ and $v_{1,\infty}, v_{2,\infty}$ respectively, with  $(u_\infty, v_{1,\infty}, v_{2,\infty}) \in \mc S$. Moreover, because the $H^{1}$ norm is lower-semicontinuous for the weak $H^1$ convergence, we have
\begin{multline*}
\int_{I_R}D \vert u_{\infty}^{\prime} \vert^{2}  +\sum_{i=1,2} \beta_i \int_{\Omega_{R}} \left( d_i\vert \nabla v_{i,\infty} \vert^{2} -a_iv_{i,\infty}^{2} \right) \leq \\ \liminf_{n \to +\infty}\left(\int_{I_R}D \vert u_n^{\prime} \vert^{2}  + \sum_{i=1,2}\beta_i \int_{\Omega_{R}} \left( d_i\vert \nabla v_{i,n} \vert^{2} - a_i v_{i,n}^{2} \right)\right).
\end{multline*}
Now, it remains to observe that, for $i=1,2$,
$$
 \int_{I_R}(\mu_i u_n- \nu_i v_{i,n}\vert_{y=0})^{2} \underset{n \to +\infty}{\longrightarrow}  \int_{I_R}(\mu_i u_{\infty}- \nu_i v_{i,\infty}\vert_{y=0})^{2},
$$
Indeed, by a classical result of Gagliardo \cite{G}, the trace operator is continuous from $H^1(\O_R)$ to $H^{\frac{1}{2}}(\partial \O_R)$, and the space $H^{\frac{1}{2}}$ is compactly imbedded in $L^{2}$ by the fractional Rellich-Kondrachov theorem, see, e.g, \cite[Theorem 9.16]{Brezis}. Therefore, $(u_\infty, v_{1,\infty},v_{2,\infty})$ is a minimizer of $\mc E$ over $\mc S$.

\medskip
\emph{Step 2. Equation satisfied by the minimizers.}\\
Now, it remains to prove that the minimizers of $\mc E$ over $\mc S$ are eigenvectors associated with the principal eigenvalue. To do so, we set $\alpha_i := \frac{\mu_i}{\mu_1 + \mu_2}$ and we compute the derivative of $\mc E$ in the direction of
a triplet $(\phi, \psi_1, \psi_2)$ of smooth test functions: 
\begin{multline*}
\lim_{\e \to 0}\frac{\mc E(u+\e\phi , v_1 + \e \psi_1, v_2 + \e \psi_2)  - \mc E(u,v_1,v_2)}{2\e}\\ 
 = \sum_i- \alpha_i \int_{I_R}D  u^{\prime\prime} \phi  - \beta_i\int_{\Omega_{R}} \left( d_i \Delta v_i  + a_i v_i \right)\psi_i - \beta_i\int_{I_R} d_i \psi_i \partial_{y}v_i \vert_{y=0} \\ + \gamma_i \int_{I_R}(\mu_i u- \nu_i v_i\vert_{y=0})(\mu_i \phi- \nu_i \psi_i\vert_{y=0}) \\
 =  \sum_i \int_{I_R}( -\alpha_i D  u^{\prime\prime} + \gamma_i \mu_i(\mu_i u  - \nu_i v_i\vert_{y=0})) \phi  + \beta_i\int_{\Omega_{R}} \left( -d_i \Delta v_i  - a_i v_i \right)\psi_i \\+\int_{I_R}( \beta_i d_i\partial_{\nu}v_i\vert_{y=0} -  \nu_i \gamma_i (\mu_i u +\nu_i v_i\vert_{y=0}))\psi_i\vert_{y=0}  \\
 =   \int_{I_R}(  -D  u^{\prime\prime} +\sum_i(\mu_i u  - \nu_i v_i\vert_{y=0})) \phi  + \sum_i \beta_i\int_{\Omega_{R}} \left( -d_i \Delta v_i  - a_i v_i \right)\psi_i \\+\sum_i \beta_i \int_{I_R}(  d_i\partial_{\nu}v_i\vert_{y=0} -   (\mu_i u +\nu_i v_i\vert_{y=0}))\psi_i\vert_{y=0}. \\
 =   -\int_{I_R} L_0(u,v_1,v_2) \phi  - \sum_i \beta_i\int_{\Omega_{R}} L_i (v_i)\psi_i -\sum_i \beta_i \int_{I_R}B_i (u,v_i)\psi_i\vert_{y=0}. 
\end{multline*}
We used the specific form of $\alpha_i, \beta_i, \gamma_i$ to get the last equality.

Now, if $(u,v_1,v_2)$ is a minimizer of $\mc E$ over $\mc S$, whose existence is guaranteed by the first step, the derivative of $\mc E$ needs to be collinear to the exterior normal to $\mc S$ at $(u,v_1,v_2)$, that is, there is $\lambda \in \R$ such that
$$
 -L_0(u,v_1,v_2) = \lambda u,
 \quad  B_i (u,v_i) = 0,
 \quad   -  \beta_i L_i (v_i) = \lambda \beta_i v_i ,\quad\text{for }i=1,2. 
$$
Therefore, the minimizer $(u,v_1,v_2)$ is an eigenfunction associated with the eigenvalue~$\lambda$. Plotting the principal eigenfunction $(u_R, v_{1,R}, v_{2,R})$ in the right-hand side of~\eqref{eqvariational} we find that $\lambda \leq \l^R$. It then follows from the last part of Proposition \ref{propKR} that $\lambda = \lambda_{1}^{R}$, whence the result. 
\end{proof}
\begin{req}
In the case where the two sides of the field are symmetric, that is, $d_1=d_2 = d$, $\mu_1= \mu_2 = \mu$, $\nu_1= \nu_2 =\nu$ and $a_1 \equiv a_2 \equiv a$, then  formula \eqref{eqvariational} can be written under the simpler form
$$
\lambda_{1}^{R} = \min_{\substack{(u,v) \in \tilde{\mathcal{H}}_{R} \\ (u,v) \not\equiv (0,0)}} \frac{    \mu \int_{I_R}D \vert u^{\prime} \vert^{2}  +  \nu \int_{\Omega_{R}} \left( d\vert \nabla v \vert^{2} - a v^{2} \right) + \int_{I_R}(\mu u- \nu v\vert_{y=0})^{2}   }{ \mu \int_{I_R}u^{2} +\nu \int_{\Omega_{R}}v^{2} }
$$
where
$$
\tilde{\mathcal{H}}_{R} :=   H^{1}_{0}(I_R) \times  H^{1}_{0}(\Omega_{R}\cup (I_{R}\times\{0\})).
$$
\end{req}

\subsection{Convergence of $\lambda_{1}^{R}$ to $\lambda_{1}$}\label{cv vp}

The existence of a generalized principal eigenfunctions triplet (\thm{gpef}) associated with $\l$ and the 
convergence of $\lambda_{1}^{R}$ to $\lambda_{1}$ (\thm{truncated}) are proved simultaneously.
We recall that the generalized principal eigenvalue $\l$ is defined by \eqref{fgpe}. We will make use of the Harnack inequality, Theorem \ref{Harnack}, whose proof is postponed to the next section.


\begin{proof}[Proofs of Theorems \ref{thm:gpef} and \ref{thm:truncated}] 
	The proofs are divided into several steps.
	
	\emph{Step 1. $ \lambda_{1} \leq \lim_{R \to +\infty}\lambda_{1}^{R}$.} \\
Proposition \ref{varform} implies that $\lambda_{1}^{R}$ is non-increasing with respect to $R$: if $R^{\prime}>R$, we can use the principal eigenfunctions $(u_{R^{\prime}},v_{1,R^{\prime}}, v_{2,R^{\prime}})$ associated with $\l^{R'}$ as ``test functions" in \eqref{fgpe} restricted to $\Omega_{R}$ and $I_R$ to get $\lambda_{1}^{R}\geq \lambda_{1}^{R^{\prime}}$. Likewise, we see that $\l^R \geq \l$, for every $R>0$. Thanks to Proposition \ref{bound lambda}, we see that the sequence $(\l^R)_{R>0}$ converges as $R$ goes to $+\infty$ to some limit $\ol{\lambda} \in \R$, that satisfies
%
$$
 \overline{\lambda} \geq \lambda_{1}.
$$
%
%
To prove the reverse inequality, we actually show that there is a triplet of positive functions $(u,v_1, v_2)$, with $u \in W^{2,p}_{loc}(\R)$ and $v_1, v_2 \in W^{2,p}_{loc}(\ol{\R \times \R^{+}})$, such that $-L_{0}(u,v_1,v_2)=\overline{\lambda}u$, $-L_{i}(v_i)=\overline{\lambda}v_i$ and $B_i(u,v_i) = 0$, for $i=1,2$. To do so, the idea is to consider the family of eigenfunctions $\left((u_{R},v_{1,R},v_{2,R})\right)_{R>0}$, provided by Proposition \ref{propKR}, and to extract a converging subsequence as $R$ goes to $+\infty$, using the $L^{p}$ elliptic estimates. To start with, let us normalize this sequence so that 
\begin{equation}\label{normalization}
u_{R}(0) + v_{1,R}(0,0) + v_{2,R}(0,0) = 1.
\end{equation}
The core of the proof consists in deriving the following estimates, for all $M>0$:
\begin{equation}\label{W u}
\| u_{R} \|_{W^{2,p}(I_M)} \leq C \quad \text{ for } R> M+2,
\end{equation}
and 
\begin{equation}\label{W v}
\| v_{i,R} \|_{W^{2,p}(\O_M)} \leq C  \quad \text{ for } R> M+2, \ i=1,2.
\end{equation}
In the whole proof, $C$ denotes a generic constant independent of $R$ (but depending of $M$).

\medskip
\emph{Step 2. Bound on $\| u_{R} \|_{W^{2,p}(I_M)}$.}\\
Let $M>0$ be given. From to our Harnack estimate, Theorem \ref{Harnack}, we know that there is a constant $C$, independent of $R$, such that, for any $R>M+2$,
\begin{equation}\label{Harnack 1}
\max\left\{ \sup_{I_{M+1}} u_{R} ,  \sup_{  \Omega_{M+1}} v_{1,R}, \sup_{  \Omega_{M+1}} v_{2,R} \right\} \leq C \min\left\{ \inf_{I_{M+1}} u_{R} , \inf_{\Omega_{M+1}} v_{1,R}, \inf_{\Omega_{M+1}} v_{2,R}\right\}\leq C.
\end{equation}
We can apply the interior $W^{2,p}$ estimates to $u_{R}$ on $I_{M} \subset I_{M+1}$, to get
$$
\|u_{R} \|_{W^{2,p}( I_{M})} \leq C(\| v_{1,R}\vert_{y=0} \|_{L^{p}(I_{M+1})} + \| v_{2,R}\vert_{y=0} \|_{L^{p}(I_{M+1})} +\| u_{R} \|_{L^{p}(I_{M+1})}).
$$ 
In view of the Harnack inequality \eqref{Harnack 1}, we derive \eqref{W u}.

\medskip
\emph{Step 3. Bound on $\| v_{i,R} \|_{W^{2,p}(\O_M)}$.}\\
Because of the non-classical boundary conditions, we cannot apply directly the $W^{2,p}$ estimates to $v_{1,R}, v_{2,R}$ on $\O_M$. We will make the boundary condition ``disappear" using a change of function and a reflexion across the road, as in the proof of Lemma~\ref{WH}. We derive the estimate for $v_{1,R}$, the case of $v_{2,R}$ being similar. First, denoting $L~:=~L_{1}~+~\lambda_{1}^{R}$, we have
\begin{equation*}
\left\{
\begin{array}{rll}
-L v_{1,R} &= 0 \quad &\text{ on }\ \O_{M+1},  \\
-d_1\partial_{y}v_{1,R}\vert_{y=0} + \nu_1 v_{1,R}\vert_{y=0} &= \mu_1 u_R \quad &\text{ on } \ I_{M+1}.
\end{array}
\right.
\end{equation*}
We define 
$$
\t v_{R} := v_{1,R} e^{-\frac{\nu_1}{d_1}y},
$$ 
and the conjugated operator 
$$
\t L w := e^{-\frac{\nu_1}{d_1}y}L(w e^{\frac{\nu_1}{d_1}y}),
$$
 so that
\begin{equation*}
\left\{
\begin{array}{rll}
-\t L \t v_R &= 0 \quad &\text{ on }\ \O_{M+1},  \\
-d_1\partial_{y}\t v_{R}\vert_{y=0}  &= \mu_1 u_R \quad &\text{ on }\ I_{M+1} .
\end{array}
\right.
\end{equation*}
%
Define
$$
w_{R}(x,y) := \t v_{R}(x,y) -   \frac{\mu_1}{d_1}u_R(x)y.
$$
We have
\begin{equation*}
\left\{
\begin{array}{rlr}
-\t Lw_{R} &=  \frac{\mu_1}{d_1} \t L (u_R(x) y) \quad &\text{ on }\ \O_{M+1}  \\
-\partial_{y}w_R\vert_{y=0}  &= 0 &\quad \text{ on }\ I_{M+1}.
\end{array}
\right.
\end{equation*}
We now have a Neumann-boundary problem. We extend $w_{R}$ and  $g := \frac{\mu_1}{d_1} \t L (u_R(x) y)$ by reflexion across the road, that is we set:
$$
\t w_R(x,y) := w_R(x,\vert y \vert )\quad \text{ and }\quad  \t g(x,y) := g(x,\vert y \vert)  \quad \text{ for }\ (x,y) \in B_{M+1}.
$$
Then, $\t w_R$ is a weak solution to
$$
-L^{\star} \t w_R = \t g \quad \text{ on } \ B_{M+1},
$$
where
$$
L^{\star} := d_1\Delta + 2 \nu_1 H(y) \partial_{y} +c\partial_{x} +a_1+\frac{\nu_1^2}{d_1} +\lambda_1^R,
$$ 
with $H(y) = 1$ if $y\geq0$ and $H(y)=-1$ if $y<0$. The operator $L^{\star}$ is elliptic with constant coefficients of the second order and bounded coefficients of the first and zeroth order. The $W^{2,p}$ regularity estimates applied to $\t w_R$ on $B_M $ then yield
$$
\| \t w_R \|_{W^{2,p}(B_M)} \leq C( \| \t w_R \|_{L^{\infty}(B_{M+1})} + \| \t g\|_{L^{p}(B_{M+1})}).
$$
Recalling the definition of $\t w_R$ and $\t g$, and observing that  the $L^p(B_{M+1})$ norm of $\t g$ is controlled by the norm of $u_R$, we have, for some $C$ independent of $R$,
$$
\| \t w_R \|_{W^{2,p}(B_M)} \leq C( \| u_{R} \|_{L^{\infty}(I_{M+1})} + \| v_{1,R} \|_{L^{\infty}(\O_{M+1})} + \|  u_R \|_{W^{2,p}(I_{M+1})}).
$$
From the second step and the Harnack inequality \eqref{Harnack 1}, we have
$$
\| \t w_R \|_{W^{2,p}(B_M)} \leq C.
$$
Now, recalling the definition of $\t w_R$ and using the triangular inequality, this reduces to
$$
\| v_{1,R} \|_{W^{2,p}(\O_M)} \leq C( \| w_{R} \|_{W^{2,p}(\O_M)} + \| u_{R} \|_{W^{2,p}(I_M)}) \leq C,
$$
and then \eqref{W v} holds.

\medskip
\emph{Step 4. Passing to the limit.} \\
Owing to \eqref{W u} and \eqref{W v}, the $W^{2,p}$ norms of $u_R$ and $v_{1,R}, v_{2,R}$ are bounded independently of $R$ on $\O_M$. Up to a diagonal extraction, the sequences $u_R$ and $v_{1,R}, v_{2,R}$ converge  to some $u_\infty \in W^{2,p}_{loc}(\R)$ and $v_{1,\infty}, v_{2,\infty} \in W^{2,p}_{loc}(\overline{\R\times \R^+})$ respectively, weakly in $W^{2,p}_{loc}(\R)$ and $W^{2,p}_{loc}(\ol{\R \times \R^{+}})$. Moreover, by Morrey's inequality, these convergences also hold in 
$C^{1}_{loc}(\R)$ and $C^{1}_{loc}(\ol{\R \times \R^{+}})$. 
It follows that
$$
-L_{0}(u_{\infty},v_{1,\infty}, v_{2,\infty})=\overline{\lambda}u, \quad -L_{i}(v_{i,\infty})=\overline{\lambda}v_i, \quad B_i(u_{\infty},v_{i,\infty}) = 0, \quad\text{for }i=1,2.
$$
We further know that $(u_{\infty},v_{1,\infty}, v_{2,\infty}) \geq (0,0,0)$ and that
$(u_{\infty},v_{1,\infty}, v_{2,\infty}) \not\equiv (0,0,0)$ thanks 
to the normalization \eqref{normalization}.
Using $(u_\infty , v_{1,\infty}, v_{2,\infty})$ as ``test functions" in~\eqref{fgpe}, we derive $\lambda_{1} \geq \overline{\lambda}$. 
Hence, $\overline{\lambda} = \lambda_{1}$.

This proves Theorem \ref{thm:truncated}. Furthermore, the strong positivity property of Lemma~\ref{int}
implies that $u_\infty , v_{1,\infty}, v_{2,\infty}$ are positive.
This means that $(u_\infty , v_{1,\infty}, v_{2,\infty})$ 
is a generalized principal eigenfunctions triplet associated with $\l$ in the sense of \thm{gpef}.
\end{proof}
We mention that some properties of the usual principal eigenfunctions do not hold true for the generalized principal eigenfunctions. For instance, it is a very striking fact that, for every $\lambda < \lambda_{1}$, one can find positive functions $(u,v_1,v_2)$ such that $-L_{0}(u,v_1,v_2)=\lambda u$, $-L_{i}(v_i)=\lambda v_i$ and $B_i(u,v_i) = 0$ for $i=1,2$ (this fact is proved in \cite{BR1, BR2} for equations, but the proof adapts easily to our framework), which is not possible when the domain is bounded: the principal eigenvalue is characterized as the smallest eigenvalue, and it is the only one associated with a positive eigenfunction.

\subsection{Continuity and monotonicity of the generalized principal eigenvalue}\label{m et c}
In this section, we use the Harnack inequality, Theorem \ref{Harnack} to show that $\l$ is a continuous and monotonous function of the parameters of the system. We prove Propositions \ref{continuity} and \ref{monotony}.
\begin{proof}[Proof of Proposition \ref{continuity}]
Let us show that $\l$ is a locally Lipschitz-continuous function of the parameter $d_1\in (0,+\infty)$. To do so, we fix $0<\underline{d}<\overline{d}$ and we consider two diffusion coefficients for the upper side of the field, $d_1, d'_1 \in (\underline{d}, \overline{d})$. Let $\l$ and $\lambda'_1$ denote the 
corresponding generalized principal eigenvalues, all the other parameters being fixed and equal.

	Let $(u,v_1,v_2)$ denote a (positive) generalized principal eigenfunctions triplet associated with $\lambda_{1}$, 
	in the sense of \thm{gpef}.
%
Let $(x,y)\in \R^{+} \times \R$. If $y\geq1$, we can apply the Schauder estimates and the usual Harnack estimate to $v_1$ on $B_{1}(x,y)$ to get
$$
 \vert \Delta v_1 (x,y) \vert \leq C\sup_{B_{1}(x,y)}v_1 \leq C v_1(x,y),
$$
where, here and in the sequel, $C$ denotes a constant independent of $(x,y)$ and of $d_1, d'_1$ (but that may depend on $\underline{d}, \overline{d}$). If $y < 1$, we can still apply the Schauder estimates to $v(\cdot +x , \cdot )$ on $\O_1$ and deduce
$$
\vert \Delta v_1(x,y) \vert \leq C (\sup_{\O_1} v_1 + \sup_{I_1}u).
$$
Therefore, owing to Theorem \ref{Harnack}, it holds true that 
$$
\vert \Delta v_1(x,y) \vert \leq C v_1(x,y).
$$
Hence, by the definition of $(u,v_1,v_2)$, we find that
$$
d'_1 \Delta v_1+c_1\partial_{x}v_1 +a_1v_1 + \lambda_{1} v_1 = (d'_1 - d_1) \Delta v_1 \leq  C \vert d_1 - d'_1 \vert   v_1.
$$
Similar arguments yield that
$$
d'_1 \partial_{y} v_1 \vert_{y=0} -\nu_1 v_1 \vert_{y=0} + \mu_{1}u \leq C \vert d_1 - d'_1 \vert u.
$$
Now, we denote $\t u := \frac{(\mu_{1} - C \vert d_1 - d'_1 \vert )}{\mu_1}u$. Up to decreasing $\vert d'_1 - d_1\vert$, we can assume that $0<\t u \leq u$. Therefore, denoting $ d'_2 := d_2$, the triplet $(\t u , v_1 , v_2)$ satisfies
\small
\begin{equation*}
\left\{
\begin{array}{lll}
D\t u^{\prime\prime} + c\t u^{\prime} + \sum_i \nu_i \frac{(\mu_{1} - C \vert d_1 - d'_1 \vert )}{\mu_1} v_i\vert_{y=0} - (\mu_1 + \mu_2) \t u+ \lambda_{1}\t u \leq 0, \quad &x \in  \mathbb{R}, \\
d'_i\Delta v_i + c_i\partial_{x}v_i + a_i v_i + (\lambda_{1} - C \vert d'_1 - d_1 \vert )v_i \leq 0, \quad &(x,y) \in   \mathbb{R}\times \mathbb{R}^{+},\\
d'_i\partial_{y}v_i\vert_{y=0} + \mu_i \t u  - \nu_i v_i\vert_{y=0} \leq 0, \quad &x \in  \mathbb{R}.
\end{array}
\right.
\end{equation*}
\normalsize
Using again Theorem \ref{Harnack}, we have (up to taking a larger constant $C$)
\small
\begin{equation*}
\left\{
\begin{array}{lll}
D\t u^{\prime\prime} + c\t u^{\prime} + \sum_i \nu_i  v_i\vert_{y=0} - (\mu_1 + \mu_2) \t u+ (\lambda_{1} - C \vert d'_1 - d_1 \vert )\t u \leq 0, \quad &x \in  \mathbb{R}, \\
d'_i\Delta v_i + c_i\partial_{x}v_i + a_i v_i + (\lambda_{1} - C \vert d'_1 - d_1 \vert )v_i \leq 0, \quad &(x,y) \in   \mathbb{R}\times \mathbb{R}^{+},  \\
d'_i\partial_{y}v_i\vert_{y=0} + \mu_i \t u  - \nu_i v_i\vert_{y=0} \leq 0, \quad &x \in  \mathbb{R}.
\end{array}
\right.
\end{equation*}
\normalsize
Now, taking $( \t u , v_1 , v_2)$ as a ``test function" in formula \eqref{fgpe} for $\lambda'_1$, we have, for some $C>0$,
$$
\lambda'_{1} \geq \lambda_{1} - \vert  d'_1 - d_1 \vert C.
$$
Exchanging the roles of $d'_1$ and $d_1$, we get
$$
\vert \lambda'_1  - \lambda_1  \vert \leq C \vert d'_1 - d_1 \vert,
$$
this proves that $\l$ is a locally Lipschitz-continuous function of the parameter $d_1$. The same argument works for the other parameters (using the $L^\infty$ norm for $a_1,a_2$).
\end{proof}
Before turning to the proof of Proposition \ref{monotony}, we 
derive the exponential decay of the generalized principal eigenfunctions triplet under the assumption
\eqref{condition strict}.

\begin{lemma}\label{exponential estimate}
Assume that $c=c_i=0$ and that \eqref{condition strict} holds.
Then, there are $\alpha, \beta, \gamma, \rho >0$ such that, for any $R>\rho$, 
the principal eigenfunctions triplet $(u_{R},v_{1,R}, v_{2,R})$ 
of the eigenproblem \eqref{eigborn} satisfies
\begin{multline*}
 (u_{R},v_{1,R}, v_{2,R}) \leq \\ \max\left\{\sup_{(-\rho , \rho)}u_{R},\sup_{\Omega_{\rho}}v_{1,R}, \sup_{\Omega_{\rho}}v_{2,R}\right\}e^{2(\alpha + \beta )\rho} (e^{- \alpha \vert x \vert } , \gamma e^{- \alpha \vert x \vert- \beta y},\gamma e^{- \alpha \vert x \vert- \beta y}).
\end{multline*}
\end{lemma}

\begin{proof} \emph{Step 1. Constructing a supersolution.}\\
For $\alpha,\beta,\gamma>0$, we define
$$
(U,V_1, V_2) := (e^{ \alpha x} , \gamma_1 e^{ \alpha x- \beta y} ,  \gamma_2 e^{ \alpha x- \beta y}).
$$
The idea is to use this triplet as a supersolution of \eqref{eigborn}. 
We consider the system

\begin{equation}\label{eigborn restrict}
\left\{
\begin{array}{rllrl}
-L_{0}(u,v_1,v_2) &= \lambda^{R}_{1}u, \quad &x \in \R\backslash (-\rho, \rho), \\
-L_{i}(v_i) &= \lambda^{R}_{1}v_i, \quad &(x,y) \in \R\times \R^+\backslash \Omega_{\rho}, \\
B_i(u,v_i) &= 0, \quad &x \in \R\backslash (-\rho, \rho). \\
\end{array}
\right.
\end{equation}
An easy computation shows that, for $(U,V_1,V_2)$ to be a supersolution of this system it is sufficient to have, for $i=1,2$,
\begin{equation}\label{alg syst}
\left\{
\begin{array}{rll}
-D\alpha^{2}  &\geq \nu_1\gamma_1 + \nu_2 \gamma_2  - (\mu_1 + \mu_2) +\lambda_{1}^{\rho},\\
-d_i(\alpha^{2} + \beta^{2}) &\geq \left( \sup_{\R\times\R^+ \setminus \Omega_{\rho}} a_i(x,y)\right)+\lambda_{1}^{\rho} ,\\
d_i\beta \gamma_i &\geq \mu_i -\nu_i \gamma_i.
\end{array}
\right.
\end{equation}
Take
$$
\gamma_i := \frac{ \mu_i}{d_i\beta +\nu_i} \ \text{ and } \ \alpha := \sqrt{\frac{(d_1\gamma_1 +d_2\gamma_2   ) \beta}{2D}}.
$$
With these values, the third inequality in \eqref{alg syst} is verified. We can choose $\beta$ small and $\rho$ large enough in such a way that the second one is also verified, because the left-hand side tends to $0$ as $\beta$ goes to $0$, while, recalling that $\lambda_1^\rho\to\lambda_1 $ as $\rho$ goes to $+\infty$ due to Theorem \ref{thm:truncated}, the right-hand-side is negative for $\rho$ sufficiently large due to \eqref{condition strict}.

The first inequality reduces to
$$
\frac{d_1\gamma_1 + d_2\gamma_2}{2}\geq\lambda_{1}^{\rho}.
$$
This inequality is fulfilled up to increasing $\rho$, because $\l \leq 0$ by \eqref{condition strict}. This concludes this step.

\medskip
\emph{Step 2. Estimates on the eigenfunctions.}\\
Let $\rho$, $(U,V_1, V_2)$ be as in the Step 1.  Take $R>\rho$. Because $(u_{R},v_{1,R},v_{2,R})$ is compactly supported and because $(U,V_1, V_2)$ is strictly positive, we can define
$$
M^{\star} := \inf\{ M>0 \ : \  M(U,V_1, V_2)\geq (u_{R},v_{1,R}, v_{2,R}) \}.
$$
By continuity, there is a contact point between $M^{\star}(U,V_1, V_2)$ and $(u_{R},v_{1,R}, v_{2,R})$. Observe that $M^{\star}(U,V_1, V_1)$ is a supersolution of \eqref{eigborn restrict} whereas $(u_{R},v_{1,R} , v_{2,R})$ is a subsolution of \eqref{eigborn restrict} in its domain of definition, because $\lambda_1^R\leq\lambda_1^\rho$. Thus, if the contact point occurs in the domain where these functions are respectively a supersolution and a subsolution of the same system, then the same arguments as in the proof of Lemma~\ref{int}, based on the strong maximum principle, would imply that $M^{\star}(U,V_1, V_1)$ and $(u_{R},v_{1,R} , v_{2,R})$ coincide there. This is impossible because $u_R(\pm R)=0$.

 this means that the contact point is reached either in $\overline{\Omega}_{\rho}$ or in $[-\rho, \rho]$. It follows that
$$
M^{\star} \leq  \max\{\sup_{(-\rho , \rho)}u_{R},\ \sup_{\Omega_{\rho}}v_{1,R}, \ \sup_{\Omega_{\rho}}v_{2,R}\}e^{(\alpha + \beta )2\rho},
$$
and then
\begin{multline*}
\forall R \geq \rho, \quad (u_{R},v_{1,R}, v_{2,R}) \leq \\ \max\left\{\sup_{(-\rho , \rho)}u_{R},\ \sup_{\Omega_{\rho}}v_{1,R} ,\ \sup_{\Omega_{\rho}}v_{2,R}\right\}e^{(\alpha + \beta )2\rho} (e^{\alpha  x  } , \gamma_1 e^{ \alpha  x - \beta y}, \gamma_2 e^{ \alpha  x - \beta y}).
\end{multline*}
Repeating the argument with 
$$
(U,V_1, V_2) := (e^{ -\alpha x} , \gamma_1 e^{ -\alpha x - \beta y} , \gamma_2 e^{ -\alpha x - \beta y})
$$
yields the result.
\end{proof}

We are now in a position to prove Proposition \ref{monotony}. We only prove that the eigenvalue is non-increasing with respect to the functions $a_1, a_2$, the result for $d_1, d_2$ and $D$ is proved similarly.
\begin{proof}[Proof of Proposition \ref{monotony}]
	
We let $\lambda_1$ and $\t \lambda_1$ denote the generalized principal eigenvalue associated with the functions $(a_1, a_2)$ and $(\t a_1,\t a_2)$ respectively, all the other parameters being identical. We assume that $(a_1, a_2) \geq (\t a_1,\t a_2)$ and $(a_1, a_2) \not\equiv (\t a_1,\t a_2)$. Consistently with our notations, for $R>0$, we let $\l^R$ and $\t \lambda_1^R$ denote the principal eigenvalues for the same problems restricted to the bounded domains $I_R, \O_R$.
	
We make use of the Rayleigh formula for $\lambda_{1}^R$, \eqref{eqvariational}. For $R,\e>0$, considering the principal eigenfunctions triplet $(u_R,v_{1,R}, v_{2,R})$  associated with $\lambda_{1}^{R}$, we find that
 
 \begin{align*}
\lambda_{1}^{R}&= \frac{    \int_{I_R}D \vert  u_R^{\prime} \vert^{2}  + \sum_{i=1,2}\left( \frac{\nu_i}{\mu_i}\int_{\Omega_{R}} \left( d_i\vert \nabla v_{i,R} \vert^{2} - a_i v_{i,R}^{2} \right) +  \frac{1}{\mu_i}\int_{I_R}(\mu_i  u_R- \nu_i v_{i,R}\vert_{y=0})^{2} \right)  }{ \int_{I_R} u_R^{2} +\sum_{i=1,2}  \frac{\nu_i}{\mu_i} \int_{\Omega_{R}}v_{i,R}^{2} }\\
&\geq\t \lambda_1^R+  \frac{\frac{\nu_1}{\mu_1}\int_{\O_R}  (\t a_1 - a_1 ) v_{1,R}^{2}  +\frac{\nu_2}{\mu_2}\int_{\O_R}  (\t a_2 - a_2 ) v_{2,R}^{2}   }{\int_{I_R} u_R^{2} +\sum_{i=1,2}  \frac{\nu_i}{\mu_i} \int_{\Omega_{R}}v_{i,R}^{2}} \\
&\geq\t \lambda_1^R
 \end{align*}
From Theorem \ref{thm:truncated}, we infer that $\l  \geq \t \lambda_1$. To show the strict inequality, we argue by contradiction and assume that $\l  = \t \lambda_1$. Then, the above computation yields, for $k=1,2$,
$$
 \frac{\int_{\O_R}  (b_k - a_k ) v_{k,R}^{2}   }{\int_{I_R} u_R^{2} +\sum_{i=1,2}  \frac{\nu_i}{\mu_i} \int_{\Omega_{R}}v_{i,R}^{2}}\underset{R \to +\infty}{\longrightarrow} 0.
$$
We now normalize $(u_R,v_{1,R}, v_{2,R})$ so that 
\begin{equation}\label{norm tech}
\min\{u_R(0), v_{1,R}(0,0), v_{2,R}(0,0)\} =1.
\end{equation}
In view to Lemma \ref{exponential estimate} and to Theorem \ref{Harnack}, we can find $C>0$ independent of $R$ such that
$$
\int_{I_R} u_R^{2} + \frac{\nu_1}{\mu_1} \int_{\Omega_{R}}v_{1,R}^{2} +  \frac{\nu_2}{\mu_2} \int_{\Omega_{R}}v_{2,R}^{2}< C.
$$
Hence, for $k=1,2$, there holds that
\begin{equation*}
\int_{\O_R}  (b_k - a_k ) v_{k,R}^{2}     \underset{R \to +\infty}{\longrightarrow} 0.
\end{equation*}
Then, for any given $\bar R>0$, 
	using again Theorem \ref{Harnack} and condition~\eqref{norm tech}, we can find 
	$C_{\bar R}>0$ such that, for $k=1,2$,
	$$
	C_{\bar R}\int_{\O_{\bar R}}  (b_k - a_k )\leq
	\int_{\O_{\bar R}}  (b_k - a_k ) v_{k,R}^{2}     \underset{R \to +\infty}{\longrightarrow} 0.
	$$
	This is impossible because either $b_1>a_1$ or $b_2>a_2$ in a set of positive measure. We have reached a contradiction, and the result is thereby complete.
\end{proof}

\appendix
\setcounter{section}{1}
\section*{Appendix}
\renewcommand{\theequation}{\thesection.\arabic{equation}}
\setcounter{equation}{1}
\setcounter{theorem}{0}


We show here how Proposition \ref{propKR} is derived from the Krein-Rutman theorem, stated below as Theorem \ref{th KR}. 
In the whole section, we assume that $R>0$ is fixed. We recall that we consider functions in the Sobolev spaces $W^{2,p}(I_R)$ and $W^{2,p}(\O_R)$. We will use several times the strong elliptic maximum principle for strong solutions in this section, see \cite[Theorem 9.6]{GT}.

Let us recall that we consider the eigenproblem \eqref{eigborn} over functions $(u,v_1,v_2) \in W^{2,p}(I_R)\times (W^{2,p}(\O_R))^2$. We start with proving the positivity Lemma \ref{int}.

\begin{proof}[Proof of Lemma \ref{int}]
Let $(u,v_1,v_2) \in W^{2,p}(I_R)\times (W^{2,p}(\O_R))^2$, $(u,v_1,v_2) \geq 0$, $(u,v_1,v_2) \not\equiv (0,0,0)$ be a supersolution of \eqref{eigborn}, for some $\lambda \in \R$. We consider $C^1$ representative for $u$ and $v_1,v_2$ in the sequel.

Observe first that $u,v_1,v_2 \not\equiv 0$. Indeed, $u\equiv 0$ would imply
$v_1=v_2=0$ on $I_R \times \{0\}$ and then
$\partial_{y}v_1=\partial_{y}v_2 =0$ on $I_R \times \{0\}$. Therefore, because each $v_i$ is a non-negative solution
 of an elliptic equation, the strong maximum principle would yield $v_1\equiv v_2\equiv0$, 
 which is impossible. Conversely, $v_i\equiv 0$ for some $i\in\{1,2\}$
 immediately yields $u\equiv0$ due to $B_i(u,v_i)= 0$, 
 whence $v_2\equiv 0$. This shows that $u,v_1,v_2 \not\equiv 0$.

\medskip
\emph{Step 1. Strong positivity of $u$.}\\
Because $v_1,v_2\geq 0$, we have
$$
-Du^{\prime \prime} - c u^{\prime} +(\mu_1+\mu_2 - \lambda)u \geq 0,
$$
i.e, $u$ is a non-negative supersolution of an elliptic equation. The strong elliptic maximum principle yields
$$ 
u >0 \ \text{ on }\ I_R.
$$
Moreover, assume that $u(R) =0$. Then, the Hopf Lemma yields that
$$
u^{\prime}(R)<0,
$$
and if $u(-R)=0$, we would similarly have $u^{\prime}(-R)>0$.

\medskip
\emph{Step 2. Strong positivity of $v_i$.}\\
We only deal with $v_1$, the situation for $v_2$ being symmetric. Because $v_1\geq0$ is a solution of an elliptic equation, the strong maximum principle yields that $v_1>0$ on $\Omega_{R}$. Let us check that $v_1>0$ on $I_R \times \{ 0 \}$. If there were $x \in (-R,R)$ such that $v_1(x,0)=0$, then we would have $-\partial_{y}v_1(x,0) = \mu_1 u(x) >0$ thanks to the first step, but this would yield that $v_1<0$ for some points close to $(x,0)$, which is impossible.

Now, assume that $v_1(x,y)=0$ for some $(x,y) \in \partial \O_R \backslash \ol{ I_{R} \times \{ 0\} }$. The Hopf lemma applies there, and then $(x,y)\cdot \nabla v_1(x,y) <0$. Now, let us show that we also have $(R,0) \cdot \nabla v_1(R,0) <0$. To do so, we turn $v_1$ into a function defined on the whole ball $B_R$ and supersolution of an elliptic equation. This can be done using the same arguments as in the proof of Lemma \ref{WH}. We start with calling
$$
w(x,y) := v_1(x,y) e^{-\frac{\nu_1}{d_1}y}.
$$
It follows that $-d_1\partial_{y}w\vert_{y=0} \geq \mu_1 u \geq 0 $
and therefore
\begin{equation*}
\left\{
\begin{array}{rlll}
-d_1\Delta w-2\nu_1 \partial_{y}w -c_1 \partial_{x}w-(\frac{\nu_1^2}{d_1}+a_1+\lambda)w &\geq  0 \quad &\text{on}\  \O_R, \\
-d_1\partial_{y}w\vert_{y=0} &\geq 0 \quad &\text{on} \ I_R.
\end{array}
\right.
\end{equation*}
Now, we extend $w$ to $B_R$ by symmetry, by defining $\t w(x,y) := w(x,\vert y \vert)$, for $(x,y) \in~B_R$. Consider now the operator $\t L$ defined by 
$$\tilde{L}:=d_1\Delta + 2 \nu_1 H(y) \partial_{y} +c_1\partial_{x} +\frac{\nu_1^2}{d_1}+a_1+\lambda,$$
where $H(y) = 1$ if $y\geq0$ and $H(y)=-1$ if $y<0$. Because~$\partial_{y}w\vert_{y=0}\leq0$, it is readily seen that~$\t w$ 
satisfies $-\tilde{L}\t w \geq 0$ in $\O_R$, in the weak (i.e., $W^{1,2}$) sense. Then $w$ is supersolution of an elliptic equation on $B_R$. On this domain, we can apply the Hopf Lemma at $(R,0)$ to get that the exterior normal derivative of $w$ is strictly negative, which boils down here to $(R,0) \cdot \nabla v_1(R,0) <0$.

We can prove similarly that $(-R,0) \cdot \nabla v_1(-R,0) <0$

\medskip
\emph{Step 3.Conclusion.}\\
To conclude the proof, consider $(\phi,\psi_1,\psi_2) \in \mc C_R$, and let us show that there is $\e>0$ such that
$$
(u,v_1,v_2) > \e (\phi,\psi_1,\psi_2).
$$
We start to show that there is $\e>0$ such that
\begin{equation}\label{ep v}
v_1> \e \psi_1.
\end{equation}
The situation for $v_2,\psi_2$ is similar. We argue by contradiction: assume that, for every $n\in \N$, we can find $(x_n,y_n) \in \O_R$ so that
\begin{equation}\label{n}
v_1(x_n,y_n) < \frac{1}{n}\psi_1(x_n,y_n).
\end{equation}
We can find $(\ol x, \ol y) \in \ol{\O}_R$ such that, up to extraction,
$$
(x_n,y_n) \underset{n \to +\infty}{\longrightarrow} (\ol x, \ol y).
$$
It is impossible to have $(\ol x, \ol y) \in \O_R\cup (I_R \times \{ 0 \})$ because \eqref{n} would yield
$$
v_1(\ol x, \ol y) =0,
$$
which is impossible thanks to the second step. Then, we are left to discard the case $\vert (\ol x, \ol y) \vert =R$. But, dividing \eqref{n} by $R - \vert ( x_n, y_n) \vert $ yields
$$
\frac{v_1(x_n,y_n)}{R - \vert ( x_n, y_n) \vert} < \frac{1}{n}\left(\frac{\psi_1(x_n,y_n)}{R - \vert ( x_n, y_n) \vert}\right).
$$ 
 The left-hand side goes to $(\ol x, \ol y) \cdot \nabla v (\ol x , \ol y)$ as $n$ goes to $+\infty$, while the right-hand side goes to zero, because $v_1$ and $\psi_1$ are $C^1$ up to the boundary and $\psi_1(\ol x , \ol y) =0$. This would yield that $(\ol x, \ol y) \cdot \nabla v_1 (\ol x , \ol y) =0$, but this is in contradiction with the second step, and then \eqref{ep v} holds true for $\e$ sufficiently small. The similar result for $u$ is derived similarly (this is actually easier as it is in dimension $1$), and then the result follows.
\end{proof}

Let us now state the Krein-Rutman theorem, on which relies the proof of Proposition \ref{propKR}.

\begin{theorem}\label{th KR}
Let $E$ be a real Banach space ordered by a salient cone $K$ (i.e., $K\cap (-K) = \{ 0\} $) with non-empty interior. Let $T$ be a linear compact operator. Assume that $T$ is strongly positive (i.e., $T(K\backslash \{ 0\}) \subset \inter K$). Then, there exists a unique eigenvalue $\lambda_1$ associated with some $u_1 \in K \backslash \{0\}$. Moreover, for any other eigenvalue $\lambda$, there holds
$$
\lambda_1 > \Re (\lambda).
$$
\end{theorem}
See \cite{KR} for a proof of this result. Now, let us take $(g_0,g_1,g_2) \in  \mc C_R$, where $\mc C_R$ is defined by \eqref{space}. Then, for $M>0$, consider the following problem:
\begin{equation}\label{syst KR}
\left\{
\begin{array}{rllrl}
-L_{0}(u,v_1,v_2) + M u&= g_0 \  &\text{ in }\ I_R, \\
-L_{i}(v_i) +M v_i &=g_i \ &\text{ in }\  \Omega_{R},  \\
B_i(u,v_i) &= 0 \ &\text{ in }  \ I_R. \\
\end{array}
\right.
\end{equation}
We define the cone of non-negative functions in $\mc C_R$:
$$
\mc C_R^+ := \{ (u,v_1,v_2) \in \mc C_R \ : \ (u,v_1,v_2) \geq (0,0,0) \}.
$$
We have the following technical result:
\begin{prop}\label{T}
There is $M\in \R$ large enough so that the system \eqref{syst KR} satisfies the following:

\begin{enumerate}
\item For any $(g_0,g_1,g_2) \in \mc C_R$, there is a unique solution $(u,v_1,v_2)\in W^{2,p}(I_R)\times (W^{2,p}(\O_R))^2$ of \eqref{syst KR}.

\item There is $C>0$ such that
 $$
 \| u \|_{W^{2,p}(I_{R})} + \| v_1 \|_{W^{2,p}(\O_R)} +\| v_2 \|_{W^{2,p}(\O_R)} \leq C \left(\| g_0 \|_{C^{1}(\ol{I_R})} + \|g_1\|_{C^{1}(\ol{\O_R})}+  \|g_2\|_{C^{1}(\ol{\O_R})}\right).
 $$
\item If $(g_0, g_1,g_2) \in \mc C_R^+\backslash \{(0,0,0)\}$, then $(u,v_1,v_2) \in \inter \mc C_R^+$.
\end{enumerate}

\end{prop}
Let us explain how Proposition \ref{T} is derived. The first statement can be proved as in \cite[Proposition A.1]{BRR4}. The second statement can be obtained exactly as in Lemma \ref{WH}, by extending $v$ to a solution of an elliptic equation on the whole ball $B_R$. The third statement is a consequence of Lemma \ref{int}.

\medskip
Owing to the first statement in Proposition \ref{T}, we can take $M$ large enough so that we can define the following linear operator:
\begin{equation*}
\begin{array}{rcrc}
T \ : \ & \mc C_R  &\to& \mc C_R \\
&(g_0,g_1,g_2) &\mapsto& (u,v_1,v_2),
\end{array}
\end{equation*}
where $(u,v_1,v_2)$ is the solution of \eqref{syst KR}. The second statement yields that $T$ is compact:~indeed, $W^{2,p}$ is compactly embedded in $C^{1}$. The third statement yields that $T$ is strictly positive with respect to the salient cone $\mc C_R^+$. We can then apply the Krein-Rutman Theorem \ref{th KR} to the operator $T$ to derive Proposition \ref{propKR}.

\bigskip
\noindent \textbf{Acknowledgements:}  This work has been supported by the ERC Advanced Grant 2013 n. 321186 ``ReaDi -- Reaction-Diffusion Equations, Propagation and Modelling'' held by H.~Berestycki. This work was also partially supported by the French National Research Agency (ANR), within  project NONLOCAL ANR-14-CE25-0013. R.~Ducasse is supported by the LabEx Archimède.
During this research, L.~Rossi was on academic leave from the University of Padova. This paper was completed while H. Berestycki was visiting the institute of Advanced Study of Hong Kong University of Science and Technology and its support is gratefully acknowledged.

\end{document}